\pdfoutput=1

\documentclass[a4paper,reqno]{interact}
\usepackage{amsmath,amssymb,epsfig,float,amsfonts,latexsym,color,cite,subfigure}
\usepackage[table]{xcolor}
\usepackage{graphicx}

\numberwithin{equation}{section}
\numberwithin{figure}{section}
\numberwithin{table}{section}

\definecolor{lightblue}{rgb}{0.33,0.5,0.90}
\definecolor{lightgreen}{rgb}{0.22,0.50,0.25}
\newcommand{\cblue}{}
\usepackage[colorlinks=true,breaklinks=true,linkcolor=lightgreen,citecolor=lightgreen]{hyperref}

\newcommand\cero{\boldsymbol{0}}

\newcommand\bH{\mathbf{H}}

\newcommand\bI{\mathbf{I}}

\newcommand\bV{\mathbf{V}}

\newcommand\beps{\boldsymbol{\varepsilon}}

\newcommand\bb{\boldsymbol{b}}
\newcommand\bk{\boldsymbol{k}}

\newcommand\nn{\boldsymbol{n}}

\newcommand\bsigma{\boldsymbol{\sigma}}

\newcommand\bu{\boldsymbol{u}}
\newcommand\bv{\boldsymbol{v}}
\newcommand\bw{\boldsymbol{w}}
\newcommand\bx{\boldsymbol{x}}

\newcommand\cP{\mathcal{P}}

\newcommand\RR{\mathbb{R}}
\newcommand\cT{\mathcal{T}}

\newcommand\cL{\mathcal{L}}
\newcommand\cA{\mathcal{A}}

\newcommand{\norm}[1]{\left\|#1\right\|}

\newcommand\bdiv{\mathop{\mathbf{div}}\nolimits}
\newcommand\vdiv{\mathop{\mathrm{div}}\nolimits}

\newcommand*{\bigchi}{\mbox{\Large$\chi$}}

\newcommand{\ds}{\,\mbox{d}s}

\theoremstyle{plain}
\newtheorem{theorem}{Theorem}[section]
\newtheorem{lemma}[theorem]{Lemma}

\theoremstyle{definition}

\theoremstyle{remark}
\newtheorem{remark}{Remark}

\renewenvironment{proof}{\noindent{\it Proof.}}{\hfill$\square$}

\allowdisplaybreaks

\begin{document}
\title{Well-posedness and discrete analysis for advection-diffusion-reaction in poroelastic media}
\author{
\name{Nitesh Verma\textsuperscript{a}, Bryan G\'omez-Vargas\textsuperscript{b},
	Luis Miguel De Oliveira Vilaca\textsuperscript{c}, Sarvesh Kumar\textsuperscript{a}, and Ricardo Ruiz-Baier\textsuperscript{d}\thanks{Author for correspondence: R. Ruiz-Baier. Email: {\tt ricardo.ruizbaier@monash.edu}}} 	
\affil{\small{\textsuperscript{a}Indian Institute of Space Science and Technology, Trivandrum 695 547, India. \\
\textsuperscript{b}CI$^{\,2}\!$MA and Departamento de Ingenier\'\i a Matem\' atica,
			Universidad de Concepci\' on, Casilla 160-C, Concepci\' on, Chile; and
			Secci\'on de Matem\' atica,
			Sede de Occidente, Universidad de Costa Rica, San Ram\'on, Alajuela, Costa Rica.\\ 
\textsuperscript{c}Laboratory of Artificial \& Natural Evolution (LANE), Department of Genetics and Evolution,
University of Geneva, 4 Boulevard d'Yvoy, 1205 Geneva, Switzerland.\\ 
\textsuperscript{d}Mathematical Institute,
 University of  Oxford, A. Wiles Building,
  Woodstock Road, OX2 6GG Oxford, UK; and
  Universidad Adventista de Chile, Casilla 7-D, Chill\'an, Chile.
  \cblue{Present address: School of Mathematics, Monash University, 9 Rainforest Walk, Clayton VIC 3800, Australia.}}
  }}

\date{\today}
\maketitle
		
\begin{abstract}
We analyse a PDE system modelling poromechanical processes (formulated in mixed form using the solid deformation, fluid pressure, and total pressure) interacting with diffusing and reacting solutes in the medium. We investigate the well-posedness of the nonlinear set of equations using fixed-point theory, Fredholm's alternative, a priori estimates, and compactness arguments. We also propose a mixed finite element method and demonstrate the stability of the scheme. \cblue{Error estimates are derived in suitable norms, and numerical experiments are conducted to illustrate the mechano-chemical coupling and to verify  the  theoretical rates of convergence.}
\end{abstract}
\begin{amscode}
65M60; 74F10; 35K57; 74L15.
\end{amscode}

	\section{Introduction and problem statement}
	
	\subsection{Scope of the paper}
	
	We aim at studying the spreading properties of a system of interacting
	species when the underlying medium is of a porous nature and it undergoes elastic
	deformations. The model we propose has the potential
	to deliver quantitative insight on the two-way coupling between the transport of solutes and poromechanical effects in the context of microscopic-macroscopic
	mechanobiology. Real biological tissues are conformed by living cells, and volume changes due to cell birth and death onset velocity fields and local deformation, eventually driving domain growth \cite{neville06}. Interconnectivity of the porous microstructure is in this case sufficient to accommodate fluid flowing locally. The described problem can be encountered in numerous applications not only related to cell biomechanics, and some of these are explored in \cblue{our very recent paper \cite{verma02} (including traumatic brain injury and calcium dynamics)}.

	From the viewpoint of solvability analysis of partial differential equations and/or the theoretical aspects of finite element discretisations, the relevant literature contains a few works
	specifically targeting the coupling of diffusion in deformable porous media. We mention for instance the classical works
	of Showalter \cite{showalter00} and Showalter and Momken \cite{showalter02} which employ the theory of degenerate equations in Hilbert spaces,
	or the study of Hadamard well-posedness of parabolic-elliptic systems governing chemo-poroelasticity with thermal
	effects \cite{malysheva06}. More recently, \cite{lee19} introduces mixed finite element schemes and
	stability analysis for a system of multiple-network poroelasticity, that resembles the model problem we are interested in. Also,
	in \cite{brun19} a six-field system including temperature dynamics has been rigorously analysed using linearisation tools, the Banach fixed-point
	theory and weak compactness, and piecewise continuation in time.
	As in \cite{lee19}, we also employ the three-field formulation for the Biot consolidation equations introduced in
	\cite{oyarzua16} (see also \cite{lee17}). However in the model we adopt here, we consider a two-way active transport: the poromechanical deformations
	affect the transport of the chemical species through advection and also by means of a volume-dependent modification of the reaction
	terms; and
	the solutes' concentration generate an active stress resulting in a distributed load depending linearly on the concentration gradients. \cblue{Let us point out that in a companion paper \cite{verma02} we are addressing in more detail the modelling formalisms, we perform a linear stability analysis to identify suitable ranges for the key coupling parameters, and we give a full set of numerical tests in 2D and 3D.}
	
	The coupled system is set up in mixed-primal structure, where the equations of poroelasticity have a mixed form using displacement, pressure, and a rescaled total pressure, and the advection-diffusion-reaction system is also set in primal form, solving for the species' concentrations. Then, we focus on the semidiscrete in-time formulation, rewriting the resulting scheme equivalently as a fixed-point equation \cite{A-G-RB2014,anaya18,ggr-2018}, and then, Schauder fixed point theorem \cite{A-G-RB2014,ggr-2018}, combined with Fredholm's  alternative \cite{anaya19b,gmm2007,oyarzua16}  and quasi-linear equations theory \cite{anaya18,ladyzenskaja}, are applied to establish the solvability of the introduced formulation. Consequently, the well-known MINI-elements family and continuous piecewise polynomials are \cblue{proposed} to approximate the three-field formulation, whereas Lagrange elements are introduced to approximate the concentrations. Thus, making use of the discrete inf-sup condition together with classical inequalities, we obtain the corresponding stability result for our approximation. The advantage of using this approach is that the stability results are independent of the Lam\'e constants of the solid, and this is particularly important to prevent volumetric locking.
	We further stress that the main difficulties in the present analysis (which are not present in the literature cited above) are related to the advective coupling appearing in the advection-reaction-diffusion system. In contrast \cblue{with, e.g., \cite{brun18,brun19}}, the advecting velocity in our case is that of the solid (instead of the Darcy velocity), which is not a primary variable in our formulation. This implies that an extra $1/(\Delta t)$ appears from the backward Euler time discretisation of the solid velocity, complicating the analysis of the semidiscrete and fully discrete problems.

	The remainder of this work is structured as follows. The governing equations as well as the main assumptions
	on the model coefficients will be stated in what is left of this Section.
	Then, in Section~\ref{sec:solvability} we derive a weak formulation and
	include preliminary properties of the mathematical structure of the problem.
	Well-posedness of the coupled problem is then analysed also in Section~\ref{sec:solvability}, focusing in the semidiscrete
	case.
	We proceed in Section~\ref{sec:FE} with the introduction of a locking-free finite element scheme for the discretisation of the model equations,
	based on a stabilised formulation from \cite{oyarzua16} for the consolidation system, and a conforming method for the
	advection-diffusion-reaction subsystem. The convergence of the fully-discrete method is
	\cblue{established in Section~\ref{sec:error}. The numerical verification of these convergence rates is carried out}
	by means of a simple test presented in Section~\ref{sec:results}, \cblue{where we also give an illustrative example of pattern formation and suppression of spatio-temporal patterning due to poro-mechanical loading}.
	We close with a discussion on model extensions in Section~\ref{sec:concl}.

	\subsection{Coupling poroelasticity and advection-diffusion-reaction}
	Let us consider a piece of soft material as a porous medium composed by a mixture of
	incompressible grains and interstitial fluid, whose description can be placed in the context of
	the classical Biot problem. As in \cite{oyarzua16,lee17}, we   introduce an
	auxiliary unknown $\psi$ representing the volumetric part of
	the total stress.
	In the absence of gravitational forces, and
	for a given body load $\bb(t):\Omega\to \RR^d$ and a mass source $\ell(t):\Omega\to \RR$, one seeks for each time $t\in (0,t_{\mathrm{final}}]$,
	the displacements of the  porous skeleton, $\bu^s(t):\Omega\to \RR^d$, and the pore pressure of the
	fluid, $p^f(t):\Omega\to\RR$, such that
	\begin{align}
	\biggl(c_0
	+\frac{\alpha^2}{\lambda}\biggr) \partial_t p^f -\frac{\alpha}{\lambda} \partial_t \psi
	- \frac{1}{\eta} \vdiv(\kappa  \nabla p^f )  &= \ell &\text{in $\Omega\times(0,t_{\mathrm{final}}]$},
	\label{eq1:mass}\\
	\bsigma & =2\mu \beps(\bu^s)- \psi \bI, & \text{in $\Omega\times(0,t_{\mathrm{final}}]$}, \label{eq:sigma}\\
	\psi & = \alpha p^f - \lambda \vdiv\bu^s, & \text{in $\Omega\times(0,t_{\mathrm{final}}]$}, \label{eq:defphi}\\
	-\bdiv\bsigma & = \rho\bb & \text{in $\Omega\times(0,t_{\mathrm{final}}]$}. \label{eq1:momentum}
	\end{align}
	Here $\kappa(\bx)$ is the hydraulic conductivity of the porous medium (possibly anisotropic),  $\rho$ is the
	density of the solid material, $\eta$ is the constant viscosity of the
	interstitial fluid, $c_0$ is the constrained specific storage coefficient, $\alpha$ is the Biot-Willis
	consolidation parameter, and $\mu,\lambda$ are the shear and dilation moduli associated with the
	constitutive law of the solid structure.
	
	We also
	consider the propagation of a generic species with concentration
	$w_1$, reacting with an additional species with concentration $w_2$. The problem
	can be written as follows
	\begin{align}
	\label{eq:ADR1}
	\partial_t w_1 + \partial_t \bu^s \cdot \nabla w_1 -
	\vdiv\{ D_1(\bx)\, \nabla w_1 \}
	&=  f(w_1,w_2,\bu^s) & \text{in } \Omega\times(0,t_{\mathrm{final}}],\\
	\label{eq:ADR2}
	\partial_t w_2 + \partial_t \bu^s \cdot \nabla w_2 -
	\vdiv \{ D_2(\bx)\, \nabla w_2 \}
	&=  g(w_1,w_2,\bu^s) & \text{in } \Omega\times(0,t_{\mathrm{final}}],
	\end{align}
	where $D_1,D_2$ are positive definite diffusion matrices (however we do not
	consider here cross-diffusion effects as in, e.g., \cite{anaya18,recho19}).
	In the well-posedness analysis the reaction kinetics are generic. Nevertheless, for sake of fixing ideas and
	in order to specify the coupling effects also through a stability analysis that will be conducted in \cite{verma02},
	they will be chosen as a modification to the classical model from \cite{schnack79}
	\begin{align*}
	f(w_1,w_2,\bu^s) &=  \beta_1(\beta_2 - w_1 + w_1^2w_2) + \gamma\, w_1\, \partial_t\vdiv\bu^s,
	\\
	g(w_1,w_2,\bu^s) &=\beta_1(\beta_3 -w_1^2w_2) + \gamma\, w_2\, \partial_t\vdiv\bu^s,
	\end{align*}
	where $\beta_1,\beta_2,\beta_3,\gamma$ are positive model constants.  Note that the mechano-chemical feedback (the process where
	mechanical deformation modifies the reaction-diffusion effects) is here
	assumed only through advection and an additional reaction term depending on local dilation.
	The latter term is here modulated by $\gamma>0$, thus representing
	a source for both species if the solid volume increases, otherwise the additional
	contribution is a sink for both chemicals \cite{neville06}.
	
	The
	poromechanical
	deformations are also actively influenced by microscopic tension generation. A very simple
	description is given in terms of active stresses: we assume that the total Cauchy stress
	contains a passive and an active component, where the passive part is as in
	\eqref{eq:sigma} and
	\begin{equation}\label{eq:total-stress}
	\bsigma_{\text{total}} = \bsigma + \bsigma_{\text{act}},
	\end{equation}
	where
	the active stress operates primarily on a given, constant direction
	$\bk$, and its intensity
	depends on a scalar field $r=r(w_1,w_2)$ and on a positive constant $\tau$, to be specified later on (see, e.g., \cite{jones12})
	\begin{equation}\label{eq:active-stress}
	\bsigma_{\text{act}} = -\tau\, r \bk\otimes\bk.
	\end{equation}

	In summary, the coupled system reads
	\begin{align}
	-\bdiv( 2\mu \beps(\bu^s)- \psi \bI + \bsigma_{\text{act}})& = \rho\bb & \text{in $\Omega\times(0,t_{\mathrm{final}}]$},\nonumber\\
	\biggl(c_0
	+\frac{\alpha^2}{\lambda}\biggr) \partial_t p^f -\frac{\alpha}{\lambda} \partial_t \psi
	- \frac{1}{\eta} \vdiv(\kappa  \nabla p^f ) &= \ell & \text{in $\Omega\times(0,t_{\mathrm{final}}]$},\nonumber\\
	\psi - \alpha p^f + \lambda \vdiv\bu^s &= 0 & \text{in $\Omega\times(0,t_{\mathrm{final}}]$}, \label{eq:coupled}\\
	\partial_t w_1 + \partial_t \bu^s \cdot \nabla w_1 -
	\vdiv ( D_1(\bx)\, \nabla w_1 )
	&=  f(w_1,w_2,\bu^s) & \text{in $\Omega\times(0,t_{\mathrm{final}}]$},\nonumber\\
	\partial_t w_2 + \partial_t \bu^s \cdot \nabla w_2 -
	\vdiv ( D_2(\bx)\, \nabla w_2 )
	&=  g(w_1,w_2,\bu^s) & \text{in $\Omega\times(0,t_{\mathrm{final}}]$},\nonumber
	\end{align}
	which we endow with appropriate initial data at rest
	\begin{equation}\label{eq:initial}
	w_1(0) = w_{1,0}, \quad w_2(0)=w_{2,0}, \quad \bu^s(0)= \cero, \quad p^f(0) = 0, \quad \psi(0) = 0 \quad \text{in $\Omega\times\{0\}$,}
	\end{equation}
	and boundary conditions in the following manner
	\begin{align}\label{bc:Gamma}
	\bu^s = \cero\quad \text{and} \quad \frac{\kappa}{\eta} \nabla p^f \cdot\nn = 0\qquad \qquad&\text{on $\Gamma\times(0,t_{\text{final}}]$},\\
	\label{bc:Sigma}
	[2\mu\beps(\bu^s) - \psi\,\bI + \bsigma_{\text{act}}]\nn = \cero \quad\text{and}\quad p^f=0\qquad\qquad  &\text{on $\Sigma\times(0,t_{\text{final}}]$},\\
	D_1(\bx)\nabla w_1\cdot \nn = 0 \quad\text{and}\quad D_2(\bx)\nabla w_2\cdot \nn=0  \qquad\qquad&\text{on $\partial \Omega\times(0,t_{\text{final}}]$},
	\end{align}
	where the boundary $\partial\Omega = \Gamma\cup\Sigma$  is disjointly split into $\Gamma$ and $\Sigma$
	where we prescribe clamped boundaries and zero fluid normal fluxes; and zero (total) traction together with constant fluid pressure, respectively. Moreover, zero concentrations normal fluxes are prescribed on $\partial \Omega$. We point out that, if we would like to start with a model in terms of the divergence ($\mathrm{div}(w_i\partial_t\bu^{s})$ instead of $\partial_t \bu^s \cdot \nabla w_i$ in \eqref{eq:ADR1}-\eqref{eq:ADR2}, $i\in \{1,2\}$), we need to assume zero total flux (including the advective term, see, e.g., \cite{anaya18}). Homogeneity of the boundary conditions is only assumed to simplify the exposition of the subsequent analysis.

	\section{Well-posedness analysis}\label{sec:solvability}
	\subsection{Weak formulation and a semi-discrete form}
	Let us multiply \eqref{eq:coupled} by adequate test functions and integrate by parts (in space) whenever
	appropriate. Incorporating the boundary conditions \eqref{bc:Gamma}-\eqref{bc:Sigma} as well as the definition of
	the total stress \eqref{eq:total-stress}, we end up with the
	following variational problem: For a given $t>0$, find $\bu^s(t)\in\bH^1_\Gamma(\Omega),p^f(t)\in H_\Sigma^1(\Omega),\psi(t)\in L^2(\Omega),w_1(t)\in H^1(\Omega),w_2(t)\in H^1(\Omega)$
	such that
	\begin{gather}
	2\mu\!\! \int_{\Omega} \beps(\bu^s):\beps(\bv^s) -\!\int_{\Omega}\!\psi\vdiv \bv^s  =
	\int_{\Omega} \rho\bb \cdot\bv^s \! + \! \int_{\Omega}\!\!\tau r \bk\otimes\bk :\beps(\bv^s)
	\quad \forall \bv^s\in\bH^1_\Gamma(\Omega),\nonumber\\
	\biggl(c_0+\frac{\alpha^2}{\lambda}\biggr) \int_{\Omega} \partial_t p^fq^f + \frac{1}{\eta}\int_{\Omega}\kappa \nabla p^f\cdot\nabla q^f - \frac{\alpha}{\lambda} \int_{\Omega} \partial_t\psi q^f =  \int_{\Omega} \ell q^f
	\quad\forall q^f\in H_\Sigma^1(\Omega),\nonumber\\
	-\int_{\Omega}\phi\vdiv \bu^s + \frac{\alpha}{\lambda} \int_{\Omega} p^f\phi - \frac{1}{\lambda} \int_{\Omega} \psi\phi = 0 \quad\forall\phi\in L^2(\Omega),\label{eq:weak}\\
	\int_\Omega  \partial_t w_1s_1 + \int_\Omega D_1\nabla w_1\cdot \nabla s_1 + \int_\Omega  (\partial_t\bu^s\cdot\nabla w_1)s_1 =
	\int_{\Omega} f(w_1,w_2,\bu^s)\,s_1 \quad \forall s_1 \in H^1(\Omega),\nonumber\\
	\int_\Omega  \partial_t w_2s_2 + \int_\Omega D_2\nabla w_2\cdot \nabla s_2 + \int_\Omega  (\partial_t\bu^s\cdot\nabla w_2)s_2 =
	\int_{\Omega} g(w_1,w_2,\bu^s)\,s_2 \quad \forall s_2 \in H^1(\Omega).\nonumber
	\end{gather}
	Next, let us discretise the time interval $(0,t_{\text{final}}]$ into equispaced points $t^n = n\Delta t$, and use the following general notation for the first order backward difference $\Delta t \delta_t X^{n+1}:=X^{n+1}-X^n$. In this way, we can write a semidiscrete form of \eqref{eq:weak}: From initial data $\bu^{s,0},p^{f,0},\psi^{0}, w_1^0, w_2^0$ and
	for $n=1,\ldots$, find $\bu^{s,n+1}\in\bH^1_\Gamma(\Omega),p^{f,n+1}\in H_\Sigma^1(\Omega),\psi^{n+1}\in L^2(\Omega),w_1^{n+1}\in H^1(\Omega),w_2^{n+1}\in H^1(\Omega)$ such that
	\begin{alignat}{5}
	&&   a_1(\bu^{s,n+1},\bv^s)   &&                 &\;+&\; b_1(\bv^s,\psi^{n+1})     &=&\;F_{r^{n+1}}(\bv^s)&\quad\forall \bv^s\in\bH^1_\Gamma(\Omega), \label{weak-u}\\
	\tilde{a}_2(p^{f,n+1},q^f)&\;+\;&               &&      a_2(p^{f,n+1},q^f)   &\;-&\;   \tilde{b}_2(q^f,\psi^{n+1})  &=&\;G_{\ell^{n+1}}(q^f) &\quad\forall q^f\in H_\Sigma^1(\Omega), \label{weak-p}\\
	&&b_1(\bu^{s,n+1},\phi)  &\;+\;& b_2(p^{f,n+1},\phi)&\;-&\; a_3(\psi^{n+1},\phi) &=&\; 0 &\quad\forall\phi\in L^2(\Omega), \label{weak-psi}    \\
	\tilde{a}_4(w_1^{n+1},s_1) &\;+ &\; a_4(w_1^{n+1},s_1) &\;+ &\;  c(w_1^{n+1},s_1,\bu^{s,n+1})  & & & = &\; J_{f^{n+1}}(s_1) &\quad\forall s_1 \in H^1(\Omega), \label{weak-w1}\\
	\tilde{a}_5(w_2^{n+1},s_2) &\;+ &\; a_5(w_2^{n+1},s_2) &\;+ &\;   c(w_2^{n+1},s_2,\bu^{s,n+1}) &&& =&\; J_{g^{n+1}}(s_2) &\quad\forall s_2\in H^1(\Omega), \label{weak-w2}
	\end{alignat}
	where the bilinear forms $a_1:\bH^1_\Gamma(\Omega)\times\bH^1_\Gamma(\Omega)\to \RR$,
	$a_2:H_\Sigma^1(\Omega)\times H_\Sigma^1(\Omega)\to \RR$, $a_3: L^2(\Omega)\times L^2(\Omega)\to \RR$,
	$a_4,a_5: H^1(\Omega)\times H^1(\Omega)\to \RR$,
	$b_1:\bH^1_\Gamma(\Omega)\times L^2(\Omega)\to \RR$, $b_2,\tilde{b}_2:H_\Sigma^1(\Omega)\times L^2(\Omega)\to \RR$,
	the trilinear form $c:H^1(\Omega)\times H^1(\Omega) \times\bH^1_\Gamma(\Omega)\to \RR$,
	and linear functionals $F_r:\bH^1_\Gamma(\Omega)\to\RR$ (for $r$ known), $G_{\ell}:H_\Sigma^1(\Omega)\to\RR$, $J_f,J_g: H^1(\Omega)\to\RR$ (for known $f$ and known $g$),
	satisfy the following specifications
	\begin{gather}
	a_1(\bu^{s,n+1},\bv^s) := 2\mu \int_{\Omega} \beps(\bu^{s,n+1}):\beps(\bv^s),\;\;b_1(\bv^s,\phi):= -\int_{\Omega}\phi\vdiv \bv^s,\;\;\, b_2(p^{f,n+1},\phi):= \frac{\alpha}{\lambda} \int_{\Omega} p^{f,n+1}\phi,\nonumber\\
	\tilde{a}_2(p^{f,n+1},q^f)  := \biggl(c_0+\frac{\alpha^2}{\lambda}\biggr)
	\int_{\Omega}  \delta_t p^{f,n+1} q^f, \quad a_2(p^{f,n+1},q^f)  := \frac{1}{\eta}\int_{\Omega}\kappa \nabla p^{f,n+1}\cdot\nabla q^f,\nonumber\\
	\tilde{b}_2(q^f,\psi^{n+1}):= \frac{\alpha}{\lambda} \int_{\Omega} \delta_t\psi^{n+1} q^f,\quad
	a_3(\psi^{n+1},\phi):=  \frac{1}{\lambda} \int_{\Omega} \psi^{n+1}\phi,\nonumber\\
	\tilde{a}_4(w_1^{n+1},s_1) : = \int_\Omega  \delta_t w_1^{n+1}s_1, \quad a_4(w_1^{n+1},s_1) : =  \int_\Omega D_1(\bx)\nabla w_1^{n+1}\cdot \nabla s_1, \label{eq:forms}\\
	\tilde{a}_5(w_2^{n+1},s_2)  : = \int_\Omega  \delta_tw_2^{n+1}s_2, \quad a_5(w_2^{n+1},s_2)  : = \int_\Omega D_2(\bx)\nabla w_2^{n+1}\cdot \nabla s_2,,\nonumber \\
	c(w,s,\bu^{s,n+1})  := \int_\Omega  (\delta_t\bu^{s,n+1}\cdot\nabla w)s, \quad
	F_{r^{n+1}}(\bv^s) := \rho\int_{\Omega} \bb^{n+1} \cdot\bv^s
	+\tau\int_{\Omega} r^{n+1}  \bk\otimes\bk :\beps(\bv^s), \nonumber \\
	G_{\ell^{n+1}}(q^f) :=  \int_{\Omega} \ell^{n+1}  q^f, \quad
	J_{f^{n+1}}(s_1) : = \int_{\Omega} f^{n+1} s_1,
	\quad J_{g^{n+1}}(s_2) : = \int_{\Omega} g^{n+1} s_2.
	\nonumber
	\end{gather}
	
	\subsection{Preliminaries}\label{preliminaries}
	We will consider that the initial data \eqref{eq:initial} are nonnegative and regular enough. Moreover,
	throughout the text we will assume that the anisotropic permeability $\kappa(\bx)$ and the diffusion matrices $D_1(\bx),D_2(\bx)$
	are uniformly bounded and positive definite in $\Omega$. The latter means that, there exist positive constants $\kappa_1, \kappa_2$, and $D_i^{\min}, D_i^{\max}$, $i\in \{1,2\}$, such that
	$$\cblue{\kappa_1|\bw|^2 \leq \bw^{\mathrm{t}}\kappa(\bx)\bw \leq \kappa_2|\bw|^2,\quad \mathrm{and}\quad D_i^{\min}|\bw|^2 \leq \bw^{\mathrm{t}}D_i(\bx)\bw \leq D_i^{\max} |\bw|^2  \quad \forall\bw \in\mathbb{R}^d,\;\;\forall\bx\in\Omega.}$$
	Also,
	for a fixed $\bu^s$, the
	reaction kinetics $f(w_1,w_2,\cdot), g(w_1,w_2,\cdot)$ satisfy the growth conditions
	\begin{gather}
	\nonumber
	|f(w_1,w_2,\cdot)|\leq C(1+|w_1|+|w_2|),\quad |g(w_1,w_2,\cdot)|\leq C(1+|{w_1}|+|w_2|) \quad \text{ for  $w_1,w_2\geq 0$},\\
	|m(w_1,w_2,\cdot)-m(\tilde{w}_1,\tilde{w}_2,\cdot)|\leq C(|w_1-\tilde{w}_1|+|w_2-\tilde{w}_2|)\quad \text{ for $m=f,g$},\label{eq:reaction-assumption}\\
	\nonumber f(w_1,w_2,\cdot)=f_0 \;{(\ge 0)} \quad\text{ and }\quad g(w_1,w_2,\cdot)=g_0 \;{(\ge 0)} \quad \text{ if $w_1\leq 0$ or $w_2\leq 0$},
	\end{gather}
	and given $w_1$, $w_2 \in \mathbb{R}$, the scalar field $r(w_1,w_2)$ defined in \eqref{eq:active-stress} is such that
	\begin{gather}
	\label{property-of-r}
	|r(w_1,w_2)|\leq |w_1|+|w_2|, \;\;\;|r(w_1,w_2)-r(\tilde{w}_1,\tilde{w}_2)|\leq C(|w_1-\tilde{w}_1|+|w_2-\tilde{w}_2|).
	\end{gather}
	In addition, according to \cite{oyarzua16}, the terms in \eqref{weak-u}-\eqref{weak-w2} fulfil the following continuity
	bounds
	\begin{gather}
	|a_1(\bu^s,\bv^s)|  \le 2\mu {C_{k,2}}\|\bu^s\|_{1,\Omega} \|\bv^s\|_{1,\Omega}, \quad  |a_2(p^{f},q^f)| \le \frac{\kappa_2}{\eta} \|p^f\|_{1,\Omega} \|q^f\|_{1,\Omega}, \nonumber\\
	|a_3(\psi,\phi)|  \le \lambda^{-1}\|\psi\|_{0,\Omega}\|\phi\|_{0,\Omega}, \; |a_4(w_1,s_1)| \le D_1^{\max} \|w_1\|_{1,\Omega} \|s_1\|_{1,\Omega}, \nonumber\\
	|a_5(w_2,s_2)| \le D_2^{\max} \|w_2\|_{1,\Omega} \|s_2\|_{1,\Omega}, \quad |b_1(\bv^s,\phi)|  \le {\sqrt{d}}\|\bv^s\|_{1,\Omega}\|\phi\|_{0,\Omega}, \label{eq:continuity}\\
	|b_2(q^f,\phi)| \le \alpha\lambda^{-1}\|q^f\|_{1,\Omega}\|\phi\|_{0,\Omega}, \quad
	|F_r(\bv^s)| \le \rho \|\bb \|_{0,\Omega} \|  \bv^s \|_{0,\Omega}  +\tau \sqrt{C_{k,2}}\|r\|_{0,\Omega} \|\bv^s\|_{1,\Omega}, \nonumber \\
	| G_{\ell}(q^f)|  \le \| \ell \|_{0,\Omega} \|q^f\|_{0,\Omega},\quad |J_f(s_1)|  \le \|f\|_{0,\Omega} \|s_1 \|_{0,\Omega}, \quad |J_g(s_2)|  \le \|g\|_{0,\Omega} \|s_2 \|_{0,\Omega}, \nonumber
	\end{gather}
	for all $\bu^s,\bv^s \in \bH^1_\Gamma(\Omega)$, $p^f,q^f \in H_\Sigma^1(\Omega)$,
	$w_1,w_2,s_1,s_2 \in H^1(\Omega)$, $\psi,\phi\in L^2(\Omega)$.
	We also have the following coercivity and positivity bounds
	\begin{gather}
	a_1(\bv^s,\bv^s)  \ge 2 \mu C_{k,1}\|\bv^s\|_{1,\Omega}^2 , \quad a_2(q^{f},q^f)| \ge \frac{\kappa_1c_p}{\eta} \|q^f\|_{1,\Omega}^2, \quad
	a_3(\phi,\phi)  = \lambda^{-1}\|\phi\|_{0,\Omega}^2, \nonumber \\
	a_4(s_1,s_1) \ge D_1^{\min} |s_1|^2_{1,\Omega}, \quad a_5(s_2,s_2) \ge D_2^{\min} |s_2|^2_{1,\Omega},   \label{eq:coercivity}
	\end{gather}
	for all $\bv^s \in \bH^1_\Gamma(\Omega)$, $\phi \in L^2(\Omega)$, $s_1,s_2 \in H^1(\Omega)$, $q^f\in H^1_\Sigma(\Omega)$, where above $C_{k,1}$ and $C_{k,2}$ are the positive constants satisfying
	$$C_{k,1}\Vert\bu^{s,n+1}\Vert^2_{1,\Omega}\leq \Vert\boldsymbol{\varepsilon}(\bu^{s,n+1})\Vert^2_{0,\Omega}\leq C_{k,2}\Vert\bu^{s,n+1}\Vert^2_{1,\Omega},$$
	and $c_p$ is the Poincar\'e constant.
	Moreover, the bilinear form $b_1$ satisfies the inf-sup condition (see, e.g.,  \cite{girault79}):
	For every $\phi \in L^2(\Omega)$, there exists $\beta >0$ \ such that
	\begin{equation}\label{inf-sup}
	\sup_{\bv^s \in \bH^1_\Gamma(\Omega)} \frac{b_1(\bv^s,\phi)}{\|\bv^s\|_{1,\Omega}}  \ge \beta \|\phi\|_{0,\Omega}.
	\end{equation}
	Finally, we recall an important discrete identity and introduce the discrete-in-time norm
	\begin{equation}\label{discrete-identity}
	\int_\Omega X^{n+1}\delta_t X^{n+1} = \frac{1}{2}\delta_t \Vert X^{n+1} \Vert{^2} +\frac{1}{2}\Delta t\Vert\delta_t X^{n+1}\Vert{^2},\qquad \Vert X\Vert^2_{\ell^2(V)}:=\Delta t\sum_{m=0}^{n}\Vert X^{m+1}\Vert_V^2,
	\end{equation}
	respectively, which will be useful for the subsequent analysis.
	
	\subsection{\cblue{Unique solvability of uncoupled ADR and poroelasticity problems}}\label{fixed-point-operator}
	\cblue{As in \cite{anaya18},  we define the following  adequate set which will  be used frequently in our subsequent analysis, particularly in  fixed point analysis:}  For $i=1,2$ and $\forall \; t=t_n, n=0,1,\dots N$ let
	$$ \mathcal{S}:=\mathbf{D} \times \mathbf{D},\quad \text{where} \quad \mathbf{D}:= \lbrace w_i(\bx, \cdot) \in L^2(\Omega): 0 \leq w_i(\bx,t_n) \leq e^{\cblue{-\theta} t_n} M \text{ for a.e. } \bx \in \Omega \rbrace,$$
	and where $M$ is a constant that satisfies $M \geq \sup \lbrace \|w_{1,0}\|_{\infty, \Omega} \;,\|w_{2,0}\|_{\infty,\Omega} \rbrace$, and
	\cblue{$ \theta$ is a positive constant to be specified later.} From system \eqref{weak-u}-\eqref{weak-w2} we then define two uncoupled subproblems.
	For a given concentration pair $(\hat{w}_1^{n+1}, \hat{w}_2^{n+1}) \in \mathcal{S}$, find a solution pair $(w_1^{n+1}, w_2^{n+1}) \in [H^1(\Omega)]^2$ of the following  uncoupled advection-diffusion-reaction (ADR) system:
\begin{alignat}{5}
	\tilde{a}_4(w_1^{n+1},s_1) &\;+ &\; a_4(w_1^{n+1},s_1) &\;+ &\;  c(w_1^{n+1},s_1,\bu^{s,n+1}) &  &  &  = &\; J_{f^{n+1}}(s_1) &\quad\forall s_1 \in H^1(\Omega), \nonumber \\
	\tilde{a}_5(w_2^{n+1},s_2) &\;+ &\; a_5(w_2^{n+1},s_2) &\;+ &\;   c(w_2^{n+1},s_2,\bu^{s,n+1}) && & =&\; J_{g^{n+1}}(s_2) &\quad\forall s_2\in H^1(\Omega). \label{eq:uncoupled1}
	\end{alignat}
	In the above system,  $\bu^{s,n+1}$ is the solution of the following uncoupled poroelastic problem:
	\begin{alignat}{5}
	&& a_1(\bu^{s,n+1},\bv^s)   &&                 &\;+&\; b_1(\bv^s,\psi^{n+1})     &=&\;F_{\hat{r}^{n+1}}(\bv^s)&\quad\forall \bv^s\in\bH^1_\Gamma(\Omega), \nonumber \\
	\tilde{a}_2(p^{f,n+1},q^f)&\;+\;&               &&      a_2(p^{f,n+1},q^f)   &\;-&\;
	\tilde{b}_2(q^f,\psi^{n+1})  &=&\;G_{\ell^{n+1}(q^f)} &\quad\forall q^f\in H_\Sigma^1(\Omega), \label{eq:uncoupled2}\\
	&& b_1(\bu^{s,n+1},\phi)  &\;+\;& b_2(p^{f,n+1},\phi)&\;-&\; a_3(\psi^{n+1},\phi) &=&\; 0 &\quad\forall\phi\in L^2(\Omega), \nonumber
	\end{alignat}
	for given $\hat{r}^{n+1}:=r(\hat{w}_1^{n+1}, \hat{w}_2^{n+1})$.

\cblue{In order to address the  unique solvability of  the semi-discrete system \eqref{weak-u}-\eqref{weak-w2}, first we need to show that the uncoupled problems  \eqref{eq:uncoupled1}  and \eqref{eq:uncoupled2} are well-posed. This is carried out  employing  the Fredholm  alternative approach,  and classical results commonly used for showing the well-posedness  of elliptic/parabolic equations.}
	%
   \cblue{\begin{lemma} \label{uncoupled-P}
		Assume that $(\hat{w}^{n+1}_1, \hat{w}^{n+1}_2) \in \mathcal{S}$. Then  problem \eqref{eq:uncoupled2} has a unique solution
		$$(\bu^{s,n+1}, p^{f,n+1}, \psi^{n+1}) \in \mathbb{V} := \bH^1_{\Gamma}(\Omega) \times H^1_{\Sigma}(\Omega) \times L^2(\Omega).$$
	\end{lemma}}
	\begin{proof}
		The main ideas are borrowed from \cite{oyarzua16}, which focuses on steady poromechanics, but
		possessing a similar structure to  \eqref{eq:uncoupled2}.
		In view of putting the formulation in operator form (amenable for analysis through
		the Fredholm alternative) we define, for  $\vec{\bu} = (\bu^{s,n+1}, p^{f,n+1}, \psi^{n+1}) \in \mathbb{V}, \vec{\bv} = (\bv^s,q^f,\phi) \in \mathbb{V}$, the operators
		\begin{align*}
		\langle\mathcal{A}(\vec{\bu}),\vec{v}\rangle &:= a_1(\bu^{s,n+1},\bv^s)+ b_1(\bv^s, \psi^{n+1}) - b_1(\bu^{s,n+1}, \phi) + \tilde{a}_2(p^{f,n+1},q^f) \\
		& \quad + a_2(p^{f,n+1},q^f)+ a_3(\psi^{n+1},\phi),
		\\  \langle\mathcal{K}(\vec{\bu}),\vec{\bv}\rangle &:= -b_2(p^{f,n+1}, \phi) - \tilde{b}_2(q^f, \psi^{n+1}),
		\\ \langle\mathcal{F},\vec{\bv}\rangle &:= F_{\hat{r}^{n+1}}(\bv^s) + G_{\ell^{n+1}}(q^f).
		\end{align*}
		As per the Fredholm alternative, the solvability of the operator problem $(\mathcal{A} +\mathcal{K}) \vec{\bu} = \mathcal{F}$ (which implies solvability of the uncoupled problem \eqref{eq:uncoupled2}), holds if $\mathcal{K}$ is compact, $\mathcal{A}$ is invertible and $\mathcal{A} +\mathcal{K}$ is injective.\\\\
		\noindent {\it Step 1.} $\mathcal{K}$ \textbf{ is compact:} Define an operator $\mathbb{B}_2: H^1(\Omega) \rightarrow L^2(\Omega)$ such that
		$\langle\mathbb{B}_2 (q^f), \phi\rangle:= b_2( q^f, \phi) $, that is, $\mathbb{B}_2 q^f = (\frac{\alpha}{\lambda} I) \circ i_c$ where $i_c:H^1(\Omega) \rightarrow L^2(\Omega)$ is compact using Rellich-Kondrachov Theorem and $I:L^2(\Omega) \rightarrow L^2(\Omega)$ is the identity map. It implies that $\mathbb{B}_2$ is compact, so is $\mathbb{B}_2^*$. Note that $\mathcal{K}(\vec{\bu}) = (0, \mathbb{B}_2(p^{f,n+1}), - \mathbb{B}_2^* (\delta_t \psi^{n+1}))$. Thus, $\mathcal{K}$ is compact.\\\\
		\noindent {\it Step 2.} $\mathcal{A}$ \textbf{ is invertible  and $(\mathcal{A} + \mathcal{K})$ is injective}: Assume $\bV:=\bH^1_{\Gamma}(\Omega),\; Q:= H^1_{\Sigma}(\Omega)$ and $Z:= L^2(\Omega)$. The invertibility of $\cA$ is equivalent to the existence of a unique solution to the operator problem: Given $\cL:= (\cL_1, \cL_2, \cL_3) \in \mathbb{V} $, find $\vec{\bu} \in \mathbb{V}$ such that $\cA \vec{\bu} = \cL$, which is equivalent to the two uncoupled problems:\\
		$\bullet$ Find $(\bu^{s,n+1}, \psi^{n+1}) \in \bV \times Z$ such that
		\begin{align}
		\label{uncoupledA1}
		a_1(\bu^{s,n+1},\bv^s)+ b_1(\bv^s, \psi^{n+1}) & = \cL_1(\bv^s) \quad \forall \bv^s \in \bV, \nonumber \\
		b_1(\bu^{s,n+1}, \phi) - a_3(\psi^{n+1},\phi) & = \cL_3(\phi) \quad \forall \phi \in Z,
		\end{align}
		$\bullet$ Find $ p^{f,n+1} \in Q$ such that
		\begin{align}
		\label{uncoupledA2}
		\tilde{a}_2(p^{f,n+1},q^f) + a_2(p^{f,n+1},q^f) = \cL_2(q^f) \quad \forall q^f \in Q.
		\end{align}
		The continuity and coercivity of  the bilinear forms \cblue{ $a_1(\cdot,\cdot)$} in combination with the inf-sup condition for $b_1(\cdot,\cdot)$ and the semi-positive definiteness of $a_3(\cdot,\cdot)$, ensure  the unique solvability of \eqref{uncoupledA1} (see \cite{boffi13}). Moreover, in view of the  coercivity of $a_2(\cdot,\cdot)$ and the classical result from, \cblue{e.g., \cite[Theorem 11.1.1, Remark 11.1.1]{quarteroni94}}, the  existence of a unique solution  to \eqref{uncoupledA2} can be easily shown. Therefore  $\mathcal{A}$ is invertible. Furthermore, analogously to the proof of  \cite[Lemma 2.4]{oyarzua16}, it is straightforward to show that $\mathcal{A+K}$ is one-to-one, which completes the proof.
		\end{proof}
		
	\cblue{The following two results focus on providing the continuous dependence on data for the unique solution of problem \eqref{eq:uncoupled2}. We begin with a preliminary estimate.}
	\cblue{\begin{lemma}
		Assume that $(\bu^{s,n+1}, p^{f,n+1}, \psi^{n+1}) \in \mathbb{V} $ is the unique solution given by Lemma \ref{uncoupled-P}. Then, there exists $C_2>0$, independent of $\Delta t$ and $\lambda$, such that, for each $n$,
		\begin{align}
		&\nonumber\frac{\mu C_{k,1}}{2}\Vert \bu^{s,n+1}\Vert^2_{1,\Omega} + \frac{c_0}{2}\Vert p^{f,n+1}\Vert^2_{0,\Omega} +\frac{\kappa_1c_p\Delta t}{2\eta}\sum_{m=0}^{n}\Vert p^{f,m+1}\Vert^2_{1,\Omega}\\
		\label{stability-uh-ph}&\qquad \leq C_2\Big\{\Vert\bu^{s,0}\Vert^2_{1,\Omega} + \Vert p^{f,0}\Vert^2_{0,\Omega}+ \Vert \psi^{0}\Vert^2_{0,\Omega}+ \sum_{m=0}^{n}\Vert\psi^{m+1}\Vert_{0,\Omega}^2+\sum_{m=0}^{n}\Vert p^{f,m+1}\Vert^2_{0,\Omega}\\
		&\nonumber \qquad\qquad+\sum_{m=0}^{n}\Vert \hat{r}^{m+1}\Vert_{0,\Omega}^2+\sum_{m=0}^{n} \Vert\bb^{m+1}\Vert_{0,\Omega}^2 + \Delta t\sum_{m=0}^{n} \norm{\ell^{m+1}}^2_{0,\Omega} \Big\}.
		\end{align}
	\end{lemma}}
\begin{proof}
		We begin by taking \cblue{$\bv^s=\delta_t\bu^{s,n+1}$} in the first row of \eqref{eq:uncoupled2}, and then applying Cauchy-Schwarz and Young inequalities, to get
		\begin{align*}
		\begin{split}
		&\mu\delta_t \Vert \boldsymbol{\varepsilon}(\bu^{s,n+1}) \Vert^2_{0,\Omega} + \mu C_{k,1}\Delta t\Vert\delta_t \bu^{s,n+1}\Vert^2_{1,\Omega}\leq \frac{1}{2\delta_1}\Vert\psi^{n+1}\Vert_{0,\Omega}^2+\frac{\delta_1}{2}\Vert\delta_t\bu^{s,n+1}\Vert_{1,\Omega}^2 \\
		&\qquad\qquad\qquad\qquad+ \frac{\tau^2}{2\delta_2}\Vert \hat{r}^{n+1} \Vert_{0,\Omega}^2+\frac{C_{k,2}\delta_2}{2}\Vert\delta_t\bu^{s,n+1}\Vert_{1,\Omega}^2 + \frac{\rho^2}{2\delta_3}\Vert\bb^{n+1}\Vert_{0,\Omega}^2+\frac{\delta_3}{2}\Vert\delta_t\bu^{s,n+1}\Vert_{1,\Omega}^2.
		\end{split}
		\end{align*}
		Next, defining $\delta_1:= \frac{\mu C_{k,1}\Delta t}{2}, \delta_2:=\frac{\mu C_{k,1}\Delta t}{2C_{k,2}}$ and $\delta_3:=\frac{\mu C_{k,1}\Delta t}{2}$, and then, multiplying the resulting inequality by $\Delta t$ and summing over $n$, we finally obtain
		\begin{align}\label{stability-uh-final}
		\begin{split}
		&\mu C_{k,1}\Vert \bu^{s,n+1}\Vert^2_{1,\Omega} + \frac{\mu C_{k,1}\Delta t^2}{4}\sum_{m=0}^{n}\Vert\delta_t \bu^{s,m+1}\Vert^2_{1,\Omega} \\
		&\qquad\qquad  \leq  C_1\Big\{\Vert\bu^{s,0}\Vert^2_{1,\Omega} + \sum_{m=0}^{n}\Vert\psi^{m+1}\Vert_{0,\Omega}^2+\sum_{m=0}^{n}\Vert \hat{r}^{m+1}\Vert_{0,\Omega}^2+\sum_{m=0}^{n} \Vert\bb^{m+1}\Vert_{0,\Omega}^2\Big\},
		\end{split}
		\end{align}
		where $C_1$ is a constant depending on $\mu, C_{k,1},C_{k,2}, \rho$, and $ \tau$.
		On the other hand, by taking $q^f=p^{f,n+1}$ and $\phi=\delta_t\psi^{n+1}$ in the second and third equation of \eqref{eq:uncoupled2}, respectively, we get
		\begin{align}
		\nonumber
		&\frac{1}{2\lambda}\delta_t\Vert\psi^{n+1}\Vert^2_{0,\Omega} + \frac{\Delta t}{2\lambda}\Vert\delta_t\psi^{n+1}\Vert^2_{0,\Omega} + \frac{1}{2}\Big(c_0+\frac{\alpha^2}{\lambda}\Big)\Big(\delta_t\Vert p^{f,n+1}\Vert^2_{0,\Omega} + \Delta t \Vert \delta_t p^{f,n+1}\Vert^2_{0,\Omega}\Big)+\frac{\kappa_1}{\eta}|p^{f,n+1}|^2_{1,\Omega}\\
		& \qquad\qquad \qquad \leq \frac{2\alpha}{\lambda}\Vert p^{f,n+1}\Vert_{0,\Omega}\Vert \delta_t \psi^{n+1}\Vert_{0,\Omega} + \norm{\ell^{n+1}}_{0,\Omega}\Vert p^{f,n+1}\Vert_{0,\Omega}- \int_{\Omega}\delta_t \psi^{n+1} \mathrm{div}\,\bu^{s,n+1}.\label{preliminar-stab-p-psi}
		\end{align}
		\cblue{Rewriting the first term on the right-hand side as $$\frac{2\alpha}{\lambda}\Vert p^{f,n+1}\Vert_{0,\Omega}\Vert \delta_t \psi^{n+1}\Vert_{0,\Omega} = 2 \bigg(\frac{1}{\sqrt{\lambda}} \Vert \delta_t \psi^{n+1}\Vert_{0,\Omega}\bigg) \bigg(\frac{\alpha}{\sqrt{\lambda}} \Vert p^{f,n+1}\Vert_{0,\Omega} \bigg),$$ and then employing the} Young's inequality in the first two terms on the right-hand side of \eqref{preliminar-stab-p-psi}, we obtain
		\begin{align*}
		\begin{split}
		&\frac{1}{2\lambda}\delta_t\Vert\psi^{n+1}\Vert^2_{0,\Omega} + \frac{\Delta t}{2\lambda}\Vert\delta_t\psi^{n+1}\Vert^2_{0,\Omega} + \frac{1}{2}\Big(c_0+\frac{\alpha^2}{\lambda}\Big)\Big(\delta_t\Vert p^{f,n+1}\Vert^2_{0,\Omega} + \Delta t \Vert \delta_t p^{f,n+1}\Vert^2_{0,\Omega}\Big)+\frac{\kappa_1}{\eta}|p^{f,n+1}|^2_{1,\Omega}\\
		& \qquad \leq \frac{\delta_1}{ \lambda} \Vert \delta_t \psi^{n+1}\Vert^2_{0,\Omega}+   \frac{\alpha^2}{ \lambda\delta_1}  \Vert p^{f,n+1}\Vert^2_{0,\Omega}+\frac{1}{2\delta_2} \norm{\ell^{n+1}}^2_{0,\Omega}+\frac{\delta_2}{2}\Vert p^{f,n+1}\Vert^2_{0,\Omega}- \int_{\Omega}\delta_t \psi^{n+1} \mathrm{div}\,\bu^{s,n+1}.
		\end{split}
		\end{align*}
		Now, choosing $\delta_1:=\frac{\Delta t}{2}$ and $\delta_2:=\frac{\kappa_1c_p}{\eta}$, and then, multiplying the resulting inequality by $\Delta t$ and summing over $n$, we deduce the following preliminar bound
		\begin{align}
		&\frac{1}{2\lambda}\Vert\psi^{n+1}\Vert^2_{0,\Omega} + \frac{1}{2}\Big(c_0+\frac{\alpha^2}{\lambda}\Big)\Big(\Vert p^{f,n+1}\Vert^2_{0,\Omega} + \Delta t^2 \sum_{m=0}^{n}\Vert \delta_t p^{f,m+1}\Vert^2_{0,\Omega}\Big)+\frac{\kappa_1c_p\Delta t}{2\eta}\sum_{m=0}^{n}\Vert p^{f,m+1}\Vert^2_{1,\Omega}\nonumber\\
		& \qquad \leq \frac{1}{2\lambda}\Vert\psi^{0}\Vert^2_{0,\Omega}+ \frac{1}{2}\Big(c_0+\frac{\alpha^2}{\lambda}\Big)\Vert p^{f,0}\Vert^2_{0,\Omega}+ \frac{2\alpha^2}{\lambda} \sum_{m=0}^{n}\Vert p^{f,m+1}\Vert^2_{0,\Omega}+\frac{\eta\Delta t}{2\kappa_1c_p}\sum_{m=0}^{n} \norm{\ell^{m+1}}^2_{0,\Omega}\label{preliminar-stab-p-psi-3}\\
		&\qquad \quad - \Delta t\sum_{m=0}^{n}\int_{\Omega}\delta_t \psi^{m+1} \mathrm{div}\,\bu^{s,m+1}.\nonumber
		\end{align}
		Finally, for the last term on the right-hand side of \eqref{preliminar-stab-p-psi-3}, we proceed similarly to \cite[Section 9]{ambartsumyan18}, applying summation by parts as well as the initial conditions \eqref{eq:initial}, to obtain that
		\begin{align*}
		\begin{split}
		&-\Delta t\sum_{m=0}^{n}\int_{\Omega}\delta_t \psi^{m+1} \mathrm{div}\,\bu^{s,m+1}= -\int_{\Omega}\psi^{n+1} \mathrm{div}\,\bu^{s,n+1}+\Delta t\sum_{m=0}^{n-1}\int_{\Omega}\psi^{m+1} \delta_t \mathrm{div}\,\bu^{s,m+1}\\
		&\qquad\leq \frac{1}{2\delta_3}\Vert\psi^{n+1}\Vert^2_{0,\Omega}+ \frac{\delta_3}{2}\Vert\bu^{s,n+1}\Vert^2_{1,\Omega}+\frac{1}{2\delta_4}\Delta t\sum_{m=0}^{n-1}\Vert\psi^{m+1}\Vert^2_{0,\Omega}+\frac{\delta_4}{2}\Delta t\sum_{m=0}^{n-1}\Vert\delta_t\bu^{s,m+1}\Vert^2_{1,\Omega},
		\end{split}
		\end{align*}
		and then, taking $\delta_3:= \mu C_{k,1}$ and $\delta_4:=\frac{\mu C_{k,1}\Delta t}{2}$, we arrive at the following estimate
		\begin{align}
		&\frac{1}{2\lambda}\Vert\psi^{n+1}\Vert^2_{0,\Omega} + \frac{1}{2}\Big(c_0+\frac{\alpha^2}{\lambda}\Big)\Big(\Vert p^{f,n+1}\Vert^2_{0,\Omega} + \Delta t^2 \sum_{m=0}^{n}\Vert \delta_t p^{f,m+1}\Vert^2_{0,\Omega}\Big)+\frac{\kappa_1c_p\Delta t}{2\eta}\sum_{m=0}^{n}\Vert p^{f,m+1}\Vert^2_{1,\Omega}\nonumber\\
		& \leq \frac{1}{2\lambda}\Vert\psi^{0}\Vert^2_{0,\Omega}+ \frac{1}{2}\Big(c_0+\frac{\alpha^2}{\lambda}\Big)\Vert p^{f,0}\Vert^2_{0,\Omega}+ \frac{2 \alpha^2}{\lambda} \sum_{m=0}^{n}\Vert p^{f,m+1}\Vert^2_{0,\Omega}+\frac{\eta\Delta t}{2\kappa_1c_p}\sum_{m=0}^{n} \norm{\ell^{m+1}}^2_{0,\Omega}\label{preliminar-stab-p-psi-4}\\
		&\, + \frac{1}{2\mu C_{k,1}}\Vert\psi^{n+1}\Vert^2_{0,\Omega}+ \frac{\mu C_{k,1}}{2}\Vert\bu^{s,n+1}\Vert^2_{1,\Omega}+\frac{1}{\mu C_{k,1}}\sum_{m=0}^{n-1}\Vert\psi^{m+1}\Vert^2_{0,\Omega}+\frac{\mu C_{k,1}\Delta t^2}{4}\sum_{m=0}^{n-1}\Vert\delta_t \bu^{s,m+1}\Vert^2_{1,\Omega}.\nonumber
		\end{align}
		Finally, the result follows after adding \eqref{stability-uh-final} and \eqref{preliminar-stab-p-psi-4}, and taking
		$$C_2:= \max\{C_1,c_0+\frac{1}{2\lambda},\frac{1}{2}(c_0+\frac{\alpha^2}{\lambda}),c_0+\frac{2 \alpha^2}{ \lambda}, \frac{\eta}{2\kappa_1c_p},\frac{2}{\mu C_{k,1}}\},$$
		where $C_2$ must be understood as a constant independent of $\lambda$, when $\lambda$ goes to infinity.
	\end{proof}
		
		\cblue{\begin{lemma}
			Assume that $(\bu^{s,n+1}, p^{f,n+1}, \psi^{n+1}) \in \mathbb{V} $ is the unique solution given by Lemma \ref{uncoupled-P}. Then, there exists $C>0$, independent of $\Delta t$ and $\lambda$, such that for each $n$,
			\begin{align}\label{eq:CD1}
			\begin{split}
			&\Vert \bu^{s,n+1}\Vert_{1,\Omega} + \sqrt{c_0}\Vert p^{f,n+1}\Vert_{0,\Omega} + \Vert \psi^{n+1}\Vert_{0,\Omega} + \Vert p^{f} \Vert_{l^2(H^1(\Omega))}\\
			&\qquad  \leq C \sqrt{\exp} \Big\{\Vert\bu^{s,0}\Vert_{1,\Omega} + \Vert p^{f,0}\Vert_{0,\Omega}+ \Vert \psi^{0}\Vert_{0,\Omega}+ \sum_{m=0}^{n} \Vert\bb^{m+1}\Vert_{0,\Omega} + \norm{\ell}_{\ell^2(L^2(\Omega))}+\sum_{m=0}^{n}\Vert \hat{r}^{m+1}\Vert_{0,\Omega} \Big\}.
			\end{split}
			\end{align}
		\end{lemma}}
		\begin{proof}
		Having established the bound given by \eqref{stability-uh-ph}, it only remains to obtain an upper bound for $\Vert\psi^{n+1}\Vert_{0,\Omega}$, independent of $\lambda$. Thus, taking $\phi=\psi^{n+1}$ in the inf-sup condition \eqref{inf-sup}, and using the first row of \eqref{eq:uncoupled2} and the continuity of $a_1$, we easily obtain
		\begin{align*}
		\begin{split}
		&\beta\Vert\psi^{n+1}\Vert_{0,\Omega}\leq \sup_{\bv^s \in \mathbf{V}} \frac{b_1(\bv^s,\psi^{n+1})}{\|\bv^s\|_{1,\Omega}}  = \sup_{\bv^s \in \mathbf{V}} \frac{-a_1(\bu^{s,n+1},\bv^{s})+F_{\hat{r}^{n+1}}(\bv^{s})}{\|\bv^s\|_{1,\Omega}}\\
		&\qquad  \leq 2\mu C_{k,2} \Vert\boldsymbol{\varepsilon}(\bu^{s,n+1})\Vert_{0,\Omega}+\sqrt{C_{k,2}}\tau \Vert \hat{r}^{n+1}\Vert_{0,\Omega}+\rho\Vert\bb^{n+1}\Vert_{0,\Omega},
		\end{split}
		\end{align*}
		or, equivalently,
		\begin{align}\label{inf-sup-psi}
		&\Vert\psi^{n+1}\Vert^2_{0,\Omega}\leq C_3\Big\{\Vert\bu^{s,n+1}\Vert^2_{1,\Omega}+ \Vert \hat{r}^{n+1}\Vert^2_{0,\Omega}+\Vert\bb^{n+1}\Vert^2_{0,\Omega}\Big\},
		\end{align}
		where $C_3$ is a constant depending on $\beta,C_{k,1},C_{k,2}, \mu, \tau$ and $\rho$. In this way, from \eqref{stability-uh-ph} and \eqref{inf-sup-psi} we finally obtain an estimate concerning the stability of the poroelasticity problem
		\begin{align}
\nonumber
		&\Vert \bu^{s,n+1}\Vert^2_{1,\Omega} + c_0\Vert p^{f,n+1}\Vert^2_{0,\Omega} + \Vert \psi^{n+1}\Vert^2_{0,\Omega} +\Delta t\sum_{m=0}^{n}\Vert p^{f,m+1}\Vert^2_{1,\Omega} \\
\label{stability-result}		&\qquad \leq C_4\Big\{\Vert\bu^{s,0}\Vert^2_{1,\Omega} + \Vert p^{f,0}\Vert^2_{0,\Omega}+ \Vert \psi^{0}\Vert^2_{0,\Omega}+ \sum_{m=0}^{n}\Vert\psi^{m+1}\Vert_{0,\Omega}^2+\sum_{m=0}^{n}\Vert p^{f,m+1}\Vert^2_{0,\Omega}\\
\nonumber		&+\sum_{m=0}^{n}\Vert \hat{r}^{m+1}\Vert_{0,\Omega}^2+\sum_{m=0}^{n} \Vert\bb^{m+1}\Vert_{0,\Omega}^2 + \Delta t\sum_{m=0}^{n} \norm{\ell^{m+1}}^2_{0,\Omega} \Big\} + C_3\Big\{\Vert \hat{r}^{n+1}\Vert^2_{0,\Omega}+\Vert\bb^{n+1}\Vert^2_{0,\Omega}\Big\}.
		\end{align}
		Finally, the stability result \eqref{eq:CD1} follows by applying Gronwall's inequality to \eqref{stability-result}.
	\end{proof}
	\begin{lemma} \label{uncoupled-ADR}
		For any $\bu^{s,n+1} \in \bV$, the uncoupled ADR system \eqref{eq:uncoupled1} has a unique solution.  Moreover there exists $C > 0$, independent of $\Delta t$, such that for each $n$, \begin{align}\label{bound-w1h-w2h-stability}
		&\Vert w_{1}^{n+1}\Vert_{0,\Omega}+\Vert w_{2}^{n+1}\Vert_{0,\Omega}+ \Vert \nabla w_{1} \Vert_{\ell^2(L^2(\Omega))}+ \Vert \nabla w_{2}\Vert_{\ell^2(L^2(\Omega))}\leq C\sqrt{\exp}\Big\{n\Delta t + \Vert w_{1}^{0}\Vert_{0,\Omega}+\Vert w_{2}^{0}\Vert_{0,\Omega}\Big \}.
		\end{align}
	\end{lemma}
	\begin{proof}
		Note that for each $n$, the  uncoupled ADR equations constitute a semilinear elliptic system; and owing to the uniform boundedness of the matrices $D_i(\bx),i=1,2$ together  with  the growth condition assumed for $f,g$; the problem \eqref{eq:uncoupled1} is uniquely solvable (see for instance,  \cite{ladyzenskaja}). On the other hand, for the continuous dependence, we begin by taking $s_{1}=w^{n+1}_{1}$ in the first equation of \eqref{eq:uncoupled1}, which yields
		\begin{align*}
		\int_\Omega  \delta_t  w_{1}^{n+1}w_{1}^{n+1} + \int_\Omega D_1(\bx)\nabla w_{1}^{n+1}\cdot \nabla w_{1}^{n+1}+  \int_\Omega  (\delta_t\bu^{s,n+1}\cdot\nabla w_{1}^{n+1})w_{1}^{n+1}=\int_{\Omega}f^{n+1}w_{1}^{n+1},
		\end{align*}
		and then, recalling that
		\begin{equation}\label{property-term}
		\int_\Omega  (\delta_t\bu^{s,n+1}\cdot\nabla w_{1}^{n+1})w_{1}^{n+1}=-\frac{1}{2}\int_{\Omega} \mathrm{div}\, ( \delta_t\bu^{s,n+1})(w_{1}^{n+1})^2,
		\end{equation}
		we can apply classical Cauchy-Schwarz inequality, to obtain
		\begin{align*}
		\begin{split}
		&\frac{1}{2}\delta_t\Vert w_{1}^{n+1}\Vert^2_{0,\Omega}+ \frac{1}{2}\Delta t \Vert \delta_t w_{1}^{n+1}\Vert^2_{0,\Omega}+  D_1^{\min} \Vert \nabla w_{1}^{n+1} \Vert^2_{0,\Omega}\\
		&\qquad\qquad \qquad\qquad\qquad\leq \frac{1}{2}\Vert\delta_t \bu^{s,n+1}\Vert_{1,\infty,\Omega}\Vert w_{1}^{n+1}\Vert^2_{0,\Omega} + \Vert f^{n+1}\Vert_{0,\Omega}\Vert w_{1}^{n+1}\Vert_{0,\Omega}.
		\end{split}
		\end{align*}
		Under the assumption that $\bu^{s,n+1}, \bu^{s,n}$ are uniformly bounded in $\mathbf{W}^{1,\infty}(\Omega)$, and after applying Young's inequality, we deduce the following result
		\begin{align*}
		\begin{split}
		&\frac{1}{2}\delta_t\Vert w_{1}^{n+1}\Vert^2_{0,\Omega}+ \frac{1}{2}\Delta t \Vert \delta_t w_{1}^{n+1}\Vert^2_{0,\Omega}+ D_1^{\min}\Vert \nabla w_{1}^{n+1}\Vert^2_{0,\Omega}\\
		&\qquad\qquad\qquad \leq \frac{C_1}{2\Delta t}\Vert w_{1}^{n+1}\Vert_{0,\Omega}^2+ \frac{1}{2}\Vert f^{n+1}\Vert^2_{0,\Omega}+\frac{1}{2}\Vert w_{1}^{n+1}\Vert^2_{0,\Omega}.
		\end{split}
		\end{align*}
		Finally, a preliminary stability result follows by summing over $n$ and multiplying by $\Delta t$, which is
		\begin{align}\label{bound-w1h}
		\begin{split}
		&\frac{1}{2}\Vert w_{1}^{n+1}\Vert^2_{0,\Omega}+ \frac{1}{2}\Delta t^2\sum_{m=0}^{n} \Vert \delta_t w_{1}^{m+1}\Vert^2_{0,\Omega}+ D_1^{\min}\Delta t \sum_{m=0}^{n}\Vert \nabla w_{1}^{m+1}\Vert^2_{0,\Omega}\\
		&\qquad \qquad \leq \frac{1}{2}\Vert w_{1}^{0}\Vert^2_{0,\Omega}+ \frac{1}{2}( C_1 + \Delta t) \sum_{m=0}^{n}\Vert w_{1}^{m+1}\Vert^2_{0,\Omega}+ \frac{\Delta t}{2}\sum_{m=0}^{n}\Vert f^{m+1}\Vert^2_{0,\Omega}.
		\end{split}
		\end{align}
		In much the same way as above, we obtain a stability result for $\Vert w_{2}^{n+1}\Vert_{0,\Omega}$
		\begin{align}\label{bound-w2h}
		\begin{split}
		&\frac{1}{2}\Vert w_{2}^{n+1}\Vert^2_{0,\Omega}+ \frac{1}{2}\Delta t^2\sum_{m=0}^{n} \Vert \delta_t w_{2}^{m+1}\Vert^2_{0,\Omega}+ \Delta t \sum_{m=0}^{n}\Vert \nabla w_{2}^{m+1}\Vert^2_{0,\Omega}\\
		&\qquad \qquad \leq \frac{1}{2}\Vert w_{2}^{0}\Vert^2_{0,\Omega}+ \frac{1}{2}(C_1+\Delta t ) \sum_{m=0}^{n}\Vert w_{2}^{m+1}\Vert^2_{0,\Omega}+ \frac{\Delta t}{2}\sum_{m=0}^{n}\Vert g^{m+1}\Vert^2_{0,\Omega},
		\end{split}
		\end{align}
		and then, from \eqref{bound-w1h} and \eqref{bound-w2h}, we get a stability bound for the uncoupled problem   \eqref{eq:uncoupled1}
		\begin{align}\label{bound-w1h-w2h}
		\begin{split}
		&\frac{1}{2}\Vert w_{1}^{n+1}\Vert^2_{0,\Omega}+\frac{1}{2}\Vert w_{2}^{n+1}\Vert^2_{0,\Omega}+ D^{\min}\Delta t \sum_{m=0}^{n}(  \Vert \nabla w_{1}^{m+1}\Vert^2_{0,\Omega}+ \Vert \nabla w_{2}^{m+1} \Vert^2_{0,\Omega}) \\
		& \qquad \qquad\qquad \qquad \leq C_2\Big\{n\Delta t+ \Vert w_{1}^{0}\Vert^2_{0,\Omega}+\Vert w_{2}^{0}\Vert^2_{0,\Omega}+\sum_{m=0}^{n}\Big(\Vert w_{1}^{m+1}\Vert^2_{0,\Omega}+\Vert w_{2}^{m+1}\Vert^2_{0,\Omega}\Big) \Big \},
		\end{split}
		\end{align}
		where we have used the growth condition on $f$ and $g$, and $D^{\min}:=\min\{D_1^{\min},D_2^{\min}\}$.
		Finally, the stability of \eqref{eq:uncoupled1} given by \eqref{bound-w1h-w2h-stability} follows from an application of Gronwall's inequality to \eqref{bound-w1h-w2h}.
	\end{proof}
	
	\subsection{\cblue{Existence of a weak solution of  fully coupled system}}
\cblue{The demonstration of the existence of a weak solution of fully coupled semi-discrete system \eqref{weak-u}-\eqref{weak-w2} relies on  fixed-point arguments. The structure of the proof requires to define  the operator $T: \mathcal{S} \rightarrow \mathcal{S}$, that for each $n$ gives
	$T(\hat{w}_1^{n+1}, \hat{w}_2^{n+1}) = (w_1^{n+1},w_2^{n+1}),$
	for a fixed pair $(\hat{w}_1^{n+1}, \hat{w}_2^{n+1}) \in \mathcal{S}$, and where $(w_1^{n+1}, w_2^{n+1}) \in [H^1(\Omega)]^2$ is the solution of \eqref{eq:newauxiliary1}-\eqref{eq:newauxiliary2} with a given displacement $\bu^{s,n+1}$ (that is,
	the solution of the uncoupled poroelastic problem \eqref{eq:uncoupled2}). Our objective is to show that $T$ has a fixed point, and as a consequence implying that the system \eqref{weak-u}-\eqref{weak-w2} possesses a weak solution. This is framed appealing to generalised Schauder's  fixed-point theorem, stated as
	\begin{lemma}\label{lem:sch}
	Let $M$ be a closed convex set in a Banach space $X$ and assume that $L: M\rightarrow M$ is a continuous mapping such that $L(M)$  is a relatively compact subset of $M$. Then $L$ has a fixed point.
	\end{lemma}
	In the context of the present problem, it is evident that $\mathcal{S}$ is a closed, bounded and convex subset of the Banach space $[L^2(\Omega)]^2$, so we further need to show that $T$ is a continuous self-map and that $T(\mathcal{S})$ is relatively compact in $S$. We dedicate the rest of this section to detail a proof of these essential steps, and we also collect other well-known required ingredients.}

\cblue{Before establishing that $T$ is a self-map, we proceed to define auxiliary
	 functions $m_{w_1}=m_{w_1}(\bx), m_{w_2}=m_{w_2}(\bx)$} in such a way that the solutions of the uncoupled ADR problem
	can be expanded as
	$$w_1= e^{\cblue{\theta} t}m_{w_1}, w_2= e^{\cblue{\theta} t}m_{w_2},$$
	for some constant $\theta >0$. Then, \cblue{since the expansion coefficients $m_{w_1},m_{w_2}$ are time-independent, it is readily seen that $w_1,w_2$ will also satisfy the auxiliary system}
	\begin{align*}
	\partial_t w_1 - \mathrm{div}\,(D_1(\bx) \nabla w_1) + \partial_t\bu^s \cdot \nabla w_1 = - \theta w_1 + e^{\cblue{-\theta} t} f(e^{\theta t} w_1, e^{\theta t} w_2),
	\\ \partial_t w_2 - \mathrm{div}\,(D_2(\bx) \nabla w_2) + \partial_t\bu^s \cdot \nabla w_2 = - \theta w_2 + e^{\cblue{-\theta} t} g(e^{\theta t} w_1, e^{\theta t} w_2),
	\end{align*}
	whose semi-discrete, variational counterpart is: Find $w_1^{n+1}, w_2^{n+1} $ such that
	\begin{align}
	\nonumber& \int_{\Omega} \delta_t w_1^{n+1} s_1 + \int_{\Omega} D_1(\bx) \nabla w_1^{n+1} \cdot \nabla s_1 + \int_{\Omega} (\delta_t\bu^{s,n+1}\cdot \nabla w_1^{n+1} ) s_1\\
	&\qquad\qquad= - \theta \int_{\Omega} w_1^{n+1} s_1 + \int_{\Omega} e^{\cblue{-\theta} t_{n+1}} f(e^{\theta t_{n+1}} w_1^{n+1}, e^{\theta t_{n+1}} w_2^{n+1})s_1 \quad \forall s_1 \in H^1(\Omega), \label{eq:auxiliary1}\\
	\nonumber & \int_{\Omega} \delta_tw_2^{n+1}s_2 + \int_{\Omega} D_2(\bx) \nabla w_2^{n+1} \cdot \nabla s_2 + \int_{\Omega} (\delta_t\bu^{s,n+1}\cdot \nabla w_2^{n+1} ) s_2 \\
	&\qquad\qquad = - \theta \int_{\Omega} w_2^{n+1} s_2 + \int_{\Omega} e^{\cblue{-\theta} t_{n+1}} g(e^{\theta t_{n+1}} w_1^{n+1}, e^{\theta t_{n+1}} w_2^{n+1}) s_2\quad \forall s_2 \in H^1(\Omega). \label{eq:auxiliary2}
	\end{align}
	The system can be equivalently stated in the form
	\begin{alignat}{5}
	\tilde{a}_4(w_1^{n+1},s_1) &\;+ &\; a_4(w_1^{n+1},s_1) &\;+ &\;  c(w_1^{n+1},s_1,\bu^{s,n+1}) &  &  &  = &\; \tilde{J}_{f^{n+1}}(s_1) &\quad\forall s_1 \in H^1(\Omega), \label{eq:newauxiliary1} \\
	\tilde{a}_5(w_2^{n+1},s_2) &\;+ &\; a_5(w_2^{n+1},s_2) &\;+ &\;   c(w_2^{n+1},s_2,\bu^{s,n+1}) && & =&\; \tilde{J}_{g^{n+1}}(s_2) &\quad\forall s_2\in H^1(\Omega), \label{eq:newauxiliary2}
	\end{alignat}
	where
	\begin{align*}
	\tilde{J}_{f^{n+1}}(s_1) &= - \theta \int_{\Omega} w_1^{n+1} s_1 + \int_{\Omega} e^{-\theta t_{n+1}} f(e^{\theta t_{n+1}} w_1^{n+1}, e^{\theta t_{n+1}} w_2^{n+1})s_1,
	\\ \tilde{J}_{g^{n+1}}(s_2) &= - \theta \int_{\Omega} w_2^{n+1} s_2 + \int_{\Omega} e^{-\theta t_{n+1}} g(e^{\theta t_{n+1}} w_1^{n+1}, e^{\theta t_{n+1}} w_2^{n+1})s_2.
	\end{align*}
\begin{lemma}
		\cblue{The operator $T$ maps $\mathcal{S}$ into itself.} 
	\end{lemma}
	\begin{proof}
		For given $(\hat{w}^{n+1}_1, \hat{w}^{n+1}_2) \in \mathcal{S}$, we need to show that $0 \leq w^{n+1}_1, w^{n+1}_2 \leq e^{ - \theta\, t_{n+1}} M$ for each $n=0, 1, \dots, N$ where  $(w^{n+1}_1, w^{n+1}_2)=T(\hat{w}^{n+1}_1, \hat{w}^{n+1}_2)$. The proof is based on induction and contradiction arguments. Given $w_{1,0} \geq 0$,  assume that $w_{1}^n \geq 0$. We then suppose that  $w_{1}^{n+1} < 0$. Setting $s_1=-(w_1^{n+1})^{-}=- \max \lbrace -w_1^{n+1},0 \rbrace$ in \eqref{eq:auxiliary1} gives us
		\begin{gather*}
		- \int_{\Omega} \left(  \frac{w_1^{n+1}-w_1^n}{\Delta t} \right) (w_1^{n+1})^{-} - \int_{\Omega} D_1(\bx) \nabla w_1^{n+1} \cdot \nabla (w_1^{n+1})^{-} - \int_{\Omega} \left( \frac{\bu^{s,n+1}- \bu^{s,n}}{\Delta t} \cdot \nabla w_1^{n+1} \right) (w_1^{n+1})^{-}  \nonumber
		\\ =\theta \int_{\Omega} w_1^{n+1} (w_1^{n+1})^{-} - \int_{\Omega} e^{-\theta t_{n+1}} f^{n+1} (w_1^{n+1})^{-},
		\\ \frac{1}{\Delta t} \int_{\Omega} ((w_1^{n+1})^{-})^2 + D_1^{\min}\int_{\Omega} (\nabla(w_1^{n+1})^{-})^2 + \int_{\Omega} \left( \frac{\bu^{s,n+1}- \bu^{s,n}}{2 \Delta t} \right)  \cdot \nabla ((w_1^{n+1})^{-})^2 + \frac{1}{\Delta t} \int_{\Omega} w_1^n (w_1^{n+1})^{-} \\
		=- \theta \int_{\Omega} ((w_1^{n+1})^{-})^2 - \int_{\Omega} e^{-\theta t_{n+1}} f^{n+1}(w_1^{n+1})^{-},
		\end{gather*}
		and therefore 	
		\begin{gather}
		\frac{1}{\Delta t} \int_{\Omega} ( (w_1^{n+1})^{-} )^2 +  D_1^{\min}\int_{\Omega} ( \nabla (w_1^{n+1})^{-})^2 -\int_{\Omega} \left( \frac{\mathrm{div}\, (\bu^{s,n+1}- \bu^{s,n})}{2 \Delta t} \right) ((w_1^{n+1})^{-})^2
		+\theta \int_{\Omega} ((w_1^{n+1})^{-})^2\nonumber\\ = - \frac{1}{\Delta t} \int_{\Omega} w_1^n (w_1^{n+1})^{-}  - \int_{\Omega} e^{-\theta t_{n+1}} (w_1^{n+1})^{-} f_0. \label{eq:aux02}
		\end{gather}
		Since $w_1^n$ and $f_0$ are non-negative, the  \cblue{right-hand side  of \eqref{eq:aux02}} is non-positive. For  $\theta \ge \frac{\|\bu^{s,n+1} - \bu^{s,n}\|_{1,\infty,\Omega}}{2 \Delta t}$ ( which is legitimate as can be seen at the end of the proof)  along with positive definiteness of $D_1(\bx)$ throughout $\Omega$ implies that $\int_{\Omega} ((w_1^{n+1})^{-})^2 \le 0$; and hence   $(w_1^{n+1})^{-} = 0$. However  $(w_1^{n+1})^{-} > 0$, which contradicts our initial assumption.  Proceeding then by induction we obtain that $w_1^{n+1} \geq 0$  for each $n$.
		The property for $w_2$ can be derived in analogous way.
		
		The other part of the inequality (that is, $w_1^n,w_2^n \leq e^{\cblue{-\theta}t_n}M$ for each $n$) follows the same lines. Given $w_{1,0} \leq M$ we assume that $w_1^n \leq e^{\cblue{-\theta} t_n} M \le e^{\cblue{-\theta} t_{n+1}}M$, and we further suppose that   $ w_1^{n+1} > e^{\cblue{-\theta}t_{n+1}}M$. Choosing   $s_1=s_1^{n+1} := (w_1^{n+1}- e^{\cblue{-\theta} t_{n+1}} M)^{+}$ in \eqref{eq:newauxiliary1}, we can readily obtain
		\begin{gather*}
		\frac{1}{\Delta t} \int_{\Omega} (w_1^{n+1}-w_1^n) s_1^{n+1} +  \int_{\Omega} D_1(\bx) \nabla w_1^{n+1} \cdot \nabla s_1^{n+1} +  \int_{\Omega} \frac{(\bu^{s,n+1} - \bu^{s,n})}{\Delta t} \cdot \nabla w_1^{n+1} s_1^{n+1} \\
		= - \theta \int_{\Omega} w_1^{n+1} s_1^{n+1}
		+ \int_{\Omega} e^{- \theta t_{n+1}} f^{n+1} s_1^{n+1},
		\end{gather*}
		which implies that
		\begin{gather*}
		\frac{1}{\Delta t} \int_{\Omega} ( s_1^{n+1})^2 +  D_1^{\min}\int_{\Omega} | \nabla s_1^{n+1}|^2 - \int_{\Omega}\frac{\mathrm{div}\,(\bu^{s,n+1} - \bu^{s,n})}{2 \Delta t} (s_1^{n+1})^2 -  \int_{\Omega} \frac{(w_1^n- e^{\cblue{-\theta} t_{n+1}} M)}{\Delta t} s_1^{n+1}
		\\ 	\leq - \theta \int_{\Omega} (s_1^{n+1})^2 -\theta  \int_{\Omega} e^{-\theta t_{n+1}}f^{n+1} s_1^{n+1}.
		\end{gather*}
		Using again that $D_1^{\min} > 0$ and the growth condition of $f$ and  $w_1^n \le e^{\cblue{-\theta} t_{n+1}} M$,
		we can assert that
		\begin{align*}
		& \frac{1}{\Delta t} \int_{\Omega} ( s_1^{n+1})^2   +
		\int_{\Omega} \left(\theta -\frac{\|\bu^{s,n+1} - \bu^{s,n}\|_{1,\infty,\Omega}}{2 \Delta t} \right) (s_1^{n+1})^2 +\theta  \int_{\Omega} e^{\cblue{-\theta} t_{n+1}} M s_1^{n+1} \\
		&\qquad \le -\theta  \int_{\Omega} e^{-\theta t_{n+1}}f^{n+1} s_1^{n+1}
		 \le C e^{-\theta t_{n+1}} \int_{\Omega} (1+ |w_1^{n+1}| + |w_2^{n+1}|) s_1^{n+1}\\
		&\qquad \le C e^{-\theta t_{n+1}} \int_{\Omega} ( |s_1^{n+1}| + |s_2^{n+1}| + (1+2 e^{\cblue{-\theta} t_{n+1}}M)) s_1^{n+1}\\
		&\qquad\le C_1 \int_{\Omega} (e^{\cblue{-\theta} t_{n+1}}M s_1^{n+1} + (s_1^{n+1})^2 + (s_2^{n+1})^2),
		\end{align*}
		and hence,  after denoting $A(\bu,\Delta t)= \frac{\|\bu^{s,n+1} - \bu^{s,n}\|_{1,\infty,\Omega}}{2 \Delta t}$, we can write the bounds
		\begin{align}\label{eq:s1}
		\frac{1}{\Delta t} \| s_1^{n+1} \|_{0,\Omega}^2 \!+\! ( \theta -\! A(\bu,\Delta t) -\! C_1) \| s_1^{n+1} \|_{0,\Omega}^2 + \!(\theta \!-\! C_1)\!\! \int_{\Omega}e^{\cblue{-\theta} t_{n+1}}M s_1^{n+1}
		\!  - \!C_1 \| s_2^{n+1} \|_{0,\Omega}^2 &\le 0,\\
		\label{eq:s2}
		\frac{1}{\Delta t} \| s_2^{n+1} \|_{0,\Omega}^2 \!+ \!( \theta \!-\! A(\bu,\Delta t) - C_2) \| s_2^{n+1} \|_{0,\Omega}^2\! +\! (\theta - C_2)\!\!\int_{\Omega} e^{\cblue{-\theta} t_{n+1}}M s_2^{n+1}
		\!- \!C_2 \| s_1^{n+1} \|_{0,\Omega}^2 & \le 0.
		\end{align}
		
		We then employ \eqref{eq:s1} and \eqref{eq:s2}, which leads to
		\begin{gather*}
		\frac{1}{\Delta t} (\| s_1^{n+1} \|_{0,\Omega}^2 + \| s_2^{n+1} \|_{0,\Omega}^2) + \left(\theta - A(\bu,\Delta t) -\max \lbrace C_1, C_2 \rbrace \right) (\|  s_1^{n+1} \|_{0,\Omega}^2 + s_2^{n+1} \|_{0,\Omega}^2)
		\\  + \left(\theta -C_1 \right) \int_{\Omega} e^{\cblue{-\theta} t_{n+1}}M s_1^{n+1} + \left(\theta -C_2 \right) \int_{\Omega} e^{\cblue{-\theta} t_{n+1}}M s_2^{n+1} \le 0,
		\end{gather*}
		and if we choose $\theta \ge A(\bu,\Delta t) + \max \lbrace C_1, C_2 \rbrace$, then we conclude, from the expression above, that  $s_1^{n+1}=s_2^{n+1}=0$. This leads to a contradiction with  $s_1^{n+1}, s_2^{n+1} > 0$, and hence   $ w_1^{n+1}, w_2^{n+1} \le e^{\cblue{-\theta} t_{n+1}}M$. An appeal to the induction principle completes the rest of the proof.
	\end{proof}
	
	\cblue{\begin{lemma}\label{cor1}
		$T(\mathcal{S})$ is relatively compact  in $[L^2(\Omega)]^2$.
	\end{lemma}}
	\begin{proof}
First we show that $T(S)$ is bounded in $[H^1(\Omega)]^2$, i.e., we need to show that $(w_1^{n+1}, w_2^{n+1}):=T(\hat{w}_1^{n+1}, \hat{w}_2^{n+1}) \in [H^1(\Omega)]^2$ for any $ (\hat{w}_1^{n+1}, \hat{w}_2^{n+1}) \in \mathcal{S}$.
		By taking  $s_1=w_1^{n+1}$ in \eqref{eq:auxiliary1} and employing  \eqref{property-term} with the definition of $\mathcal{S}$, we immediately see that
		\begin{align}\label{eq:boundforw}
		\frac{1}{\Delta t} \|w_1^{n+1}\|_{0,\Omega}^2 +  D_1^{\min}\int_{\Omega}| \nabla w_1^{n+1}|^2 = & \int_{\Omega}\frac{\mathrm{div}\,(\bu^{s,n+1} - \bu^{s,n})}{2 \Delta t} (w_1^{n+1})^2 +  \int_{\Omega} \frac{w_1^n w_1^{n+1}}{\Delta t}
		\nonumber \\  & - \theta \|w_1^{n+1}\|_{0,\Omega}^2  + \int_{\Omega} e^{- \theta t_{n+1}} f^{n+1} w_1^{n+1}.
		\end{align}
		Using the boundedness of the terms appearing in the right-hand side of  \eqref{eq:boundforw}, we have
		$$ \|w_1^{n+1}\|_{1,\Omega} \le \text{Constant},$$
		and	thus $w_1^{n+1} \in H^1(\Omega)$. Showing that  $w_2^{n+1} \in H^1(\Omega)$ is analogous. \cblue{ Now compact embedding of vector space $[H^1(\Omega)]^2$ into $[L^2(\Omega)]^2$ together with boundedness of $T(\mathcal{S})$  conclude that  $T(\mathcal{S})$ is relatively compact in $[L^2(\Omega)]^2$.}
	\end{proof}
	\begin{lemma} \label{continuous}
		The map $T$ is continuous.
	\end{lemma}
	\begin{proof}
		Let $(\hat{w}_{1,k}^{n+1}, \hat{w}_{2,k}^{n+1})_k \in \mathcal{S}$ be a sequence such that $(\hat{w}_{1,k}^{n+1}, \hat{w}_{2,k}^{n+1})_k \rightarrow (\hat{w}_1^{n+1}, \hat{w}_2^{n+1})$ in $[L^2(\Omega)]^2$ as $k\to\infty$.
		From the definition of $T$ we have that $(w_{1,k}^{n+1}, w_{2,k}^{n+1})=T(\hat{w}_{1,k}^{n+1}, \hat{w}_{2,k}^{n+1})$.

We then proceed to extract from $(\hat{w}_{1,k}^{n+1}, \hat{w}_{2,k}^{n+1})_k$ a subsequence $(\hat{w}_{1,k_j}^{n+1}, \hat{w}_{2,k_j}^{n+1})_{j}$ which   converges  to $(\hat{w}_1^{n+1}, \hat{w}_2^{n+1})$ a.e. in $\Omega$. Consequently, and owing to the continuity and boundedness of the  function, we have that  $ r(\hat{w}_{1,k_j}^{n+1}, \hat{w}_{1,k_j}^{n+1})$ converges to $r(\hat{w}_1^{n+1},\hat{w}_2^{n+1})$ in $[L^2(\Omega)]^2$.
		Moreover,	since the subsequence $(w_{1,k_j}^{n+1}, w_{2,k_j}^{n+1})_j $ is bounded in $[H^1(\Omega)]^2$, there exists a subsequence $( w_{1,(k_j)_q}^{n+1}, w_{2,(k_j)_q}^{n+1})_{q}$ such that
		$$( w_{1,(k_j)_q}^{n+1}, w_{2,(k_j)_q}^{n+1})_{q} \xrightarrow {q \rightarrow \infty}  (w_1^{n+1}, w_2^{n+1}),$$ weakly in $[H^1(\Omega)]^2$, strongly in $[L^2(\Omega)]^2$, and  a.e. in $\Omega$. And after taking the limit $q \to \infty$ in \eqref{eq:auxiliary1}- \eqref{eq:auxiliary2} \cblue{with variables $( \hat{w}_{1,(k_j)_q}^{n+1}, \hat{w}_{2,(k_j)_q}^{n+1})$}, we can assert that the converging subsequence of  $(w_{1,k}^{n+1},w_{2,k}^{n+1})_k$ in $[L^2(\Omega)]^2$ has as a limit $(w_1^{n+1},w_2^{n+1})=T(\hat{w}_1^{n+1},\hat{w}_2^{n+1})$.
		Proceeding in a similar fashion, we can safely say that all convergent subsequences of  $(w_{1,k}^{n+1},w_{2,k}^{n+1})_k$ have a unique limit \cblue{$T(\hat{w}_1^{n+1},\hat{w}_2^{n+1}) = (w_1^{n+1},w_2^{n+1})$. Using Lemma \ref{cor1} and the fact that every subsequence of $(w_{1,k}^{n+1},w_{2,k}^{n+1})_k$  has a unique limit}, we can conclude that $(w_{1,k}^{n+1}, w_{2,k}^{n+1})_k$ converges to $T(\hat{w}_1^{n+1}, \hat{w}_2^{n+1})$ in $[L^2(\Omega)]^2$.
	\end{proof}
	
	
	In view of the above results, \cblue{an application of the generalised Schauder's theorem (Lemma \ref{lem:sch}) enable us  to state the following existence theorem.}
	\begin{theorem}
		The  semi-discrete formulation  \eqref{weak-u}- \eqref{weak-w2} for problem \eqref{eq:coupled} possesses at least one solution.
	\end{theorem}	
	\subsection{Uniqueness of weak solutions}
	\cblue{In order to obtain the uniqueness of the weak solution of \eqref{weak-u}-\eqref{weak-w2}, we establish the following two preliminary results. }
	\cblue{\begin{lemma}\label{uniqueness-poro}
	Let $
	\mathcal{U}^{n+1}, \mathcal{P}^{n+1}, \bigchi^{n+1},
	\mathcal{W}_1^{n+1}$, and $\mathcal{W}_2^{n+1}$ differences between two solutions associated with the semi-discrete weak formulation \eqref{weak-u}-\eqref{weak-w2}. Then
	\begin{align} \label{bound-up}
	\| \mathcal{U}^{n+1}\|_{1,\Omega}^2 + c_0\| \mathcal{P}^{n+1}\|_{0,\Omega}^2 + \| \bigchi^{n+1}\|_{0,\Omega}^2 + \| \mathcal{P}\|^2_{l^2(H^1(\Omega))} \le  C \bigg( \|\mathcal{U}^0 \|_{1,\Omega}^2 + \|\mathcal{P}^0 \|_{0,\Omega}^2 + \|\bigchi^0 \|_{0,\Omega}^2 \nonumber\\
	+\sum_{m=0}^n \| \bb_1^{m+1} - \bb_2^{m+1} \|_{0,\Omega}^2 + \| \ell_1 - \ell_2 \|_{l^2(L^2( \Omega))}^2 + \sum_{m=0}^{n} \big( \| \mathcal{P}^{m+1}\|_{0,\Omega}^2 + \| \bigchi^{m+1}\|_{0,\Omega}^2  \\
	+ \| \mathcal{W}_1^{ m+1} \|_{0,\Omega}^2 + \|\mathcal{W}_2^{m+1} \|_{0,\Omega}^2 \big) \bigg).\nonumber
	\end{align}
	\end{lemma}}
\begin{proof}
We follow the strategy adopted in \cite{anaya18} and define two solutions $(\bu_1^{s,n+1},p_1^{f,n+1}, \psi_1^{n+1}, w_1^{1,n+1},w_2^{1,n+1})$ and $(\bu_2^{s,n+1},p_2^{f,n+1}, \psi_2^{n+1}, w_1^{2,n+1},w_2^{2,n+1})$ associated with initial data $\bb^{n+1}_1, \ell^{n+1}_1, \bu^{s,0}_1, p^{f,0}_1, \psi^0_1$, $w^1_{1,0}, w^1_{2,0}$, and $\bb^{n+1}_2, \ell^{n+1}_2,  \bu^{s,0}_2, p^{f,0}_2, \psi^0_2$, $w^2_{1,0}, w^2_{2,0}$,  respectively, and then
\begin{gather*}
\mathcal{U}^{n+1}= \bu_1^{s,n+1}- \bu_2^{s,n+1},\;\; \mathcal{P}^{n+1} = p_1^{f,n+1} - p_2^{f,n+1}, \;\; \bigchi^{n+1} = \psi_1^{n+1} - \psi_2^{n+1},\\
\mathcal{W}_1^{n+1} = w_1^{1,n+1} - w_1^{2,n+1}, \;\; \mathcal{W}_2^{n+1} = w_2^{1,n+1} - w_2^{2,n+1}.\end{gather*}
In this way, it follows from \eqref{weak-u}-\eqref{weak-psi} that
\begin{align*}
& 2 \mu \int_{\Omega} \beps(\mathcal{U}^{n+1}) : \beps (\bv^s)
- \int_{\Omega} \bigchi^{n+1} \bdiv \bv^s  - \rho \int_{\Omega} ( (\bb_1^{n+1} - \bb_2^{n+1})\cdot \bv^s \\
&\qquad   - \tau \int_{\Omega} (r_1^{n+1}-r_2^{n+1}) (k \otimes k): \beps(\bv^s)  = 0, \\
& \frac{1}{2} \left(c_0 + \frac{\alpha^2}{\lambda} \right) \int_{\Omega} \delta_{t} \cP^{n+1} q^f +  \;   \int_{\Omega} \frac{\kappa}{\eta} \nabla \cP^{n+1} \cdot \nabla q^f  - \frac{\alpha}{\lambda} \int_{\Omega} q^f \delta_{t} \bigchi^{n+1} - \int_{\Omega} \left(\ell_1^{n+1} -\ell_2^{n+1} \right) q^f = 0, \\
&-\int_{\Omega} \phi \; \mathrm{div}\, \mathcal{U}^{n+1} +  \; \frac{\alpha}{\lambda} \int_{\Omega} \mathcal{P}^{n+1} \phi  - \frac{1}{\lambda}\int_{\Omega} \bigchi^{n+1} \phi  =0,
\end{align*}
for all $\bv^s \in \bV$, all $q^f \in Q$, and all $\phi \in Z$.
Finally, similarly as in the proof of Lemma~\ref{uncoupled-P} we employ $\delta_{t} \mathcal{U}^{n+1}$, $\mathcal{P}^{n+1}$, $\delta_{t} \bigchi^{n+1}$ as test functions, together with \eqref{property-of-r} to arrive at the desired result \eqref{bound-up}.
\end{proof}

\cblue{\begin{lemma}
	Consider the hypothesis defined previously in the statement of Lemma \ref{uniqueness-poro}. Then \begin{align} \label{bound-w}
	\| \mathcal{W}_1^{n+1} \|_{0, \Omega}^2 + \| \mathcal{W}_2^{n+1} \|_{0, \Omega}^2 + \Delta t^2 \sum_{m=0}^n (\| \delta_{t}\mathcal{W}_1^{m+1} \|_{0, \Omega}^2+ \| \delta_{t}\mathcal{W}_2^{m+1} \|_{0, \Omega}^2)  + D^{\min} \Delta t \sum_{m=0}^n  (|\mathcal{W}_1^{m+1}|^2_{1,\Omega} \nonumber  \\
	+ |\mathcal{W}_2^{m+1}|^2_{1,\Omega} )\; \leq \;C \bigg( \| \mathcal{W}_1^{0} \|_{0, \Omega}^2 + \| \mathcal{W}_2^{0} \|_{0, \Omega}^2 + \| \mathcal{U}^{0} \|_{0, \Omega}^2 +
	\sum_{m=0}^n \|\mathcal{U}^{m+1}\|^2_{0,\Omega}    \\
	+  (1 + \Delta t) \sum_{m=0}^n (\|\mathcal{W}_1^{m+1}\|_{0,\Omega}^2 + \|\mathcal{W}_2^{m+1}\|_{0,\Omega}^2) \bigg). \nonumber
	\end{align}
\end{lemma}}
\begin{proof}
	We proceed analogously as in the proof of Lemma \ref{uniqueness-poro}. In fact, for the ADR problem, we can get from \eqref{weak-w1} and \eqref{weak-w2} with test functions $ \mathcal{W}_1^{n+1}$ and $\mathcal{W}_2^{n+1}$, respectively,
	the relations
	\begin{align}
	&	\frac{1}{2} \left( \delta_{t} \| \mathcal{W}_1^{n+1} \|_{0, \Omega}^2 + \Delta t \| \delta_{t}\mathcal{W}_1^{n+1} \|_{0, \Omega}^2\right) + D_1^{\min} | \mathcal{W}_1^{n+1}|^2_{1,\Omega} \nonumber \\
	&\qquad  \leq \int_{\Omega} (f_1^{n+1}-f_2^{n+1})  \mathcal{W}_1^{n+1}
	- \int_{\Omega} (\delta_{t}\bu_1^{s,n+1}\cdot \nabla \mathcal{W}_1^{n+1}
	+ \delta_{t} \mathcal{U}^{n+1}\cdot \nabla w_1^{2,n+1}) \mathcal{W}_1^{n+1},
	\cblue{\label{uniquew1}}\\
	&		\frac{1}{2} \left( \delta_{t} \| \mathcal{W}_2^{n+1} \|_{0, \Omega}^2 + \Delta t \| \delta_{t}\mathcal{W}_2^{n+1} \|_{0, \Omega}^2\right) + D_2^{\min} | \mathcal{W}_2^{n+1}|^2_{1,\Omega} \nonumber\\
	&\qquad  \leq \int_{\Omega} (g_1^{n+1}-g_2^{n+1})  \mathcal{W}_2^{n+1}	- \int_{\Omega} (\delta_{t}\bu_1^{s,n+1}\cdot \nabla \mathcal{W}_2^{n+1}
	+ \delta_{t} \mathcal{U}^{n+1}\cdot \nabla w_2^{2,n+1}) \mathcal{W}_2^{n+1}.
	\cblue{\label{uniquew2}}
	\end{align}
	As in Lemma \ref{uncoupled-ADR}, we integrate by parts \eqref{uniquew1} and assume that $w_i^{j,n+1} \in W^{1,\infty}(\Omega)$, $i,j =1,2$, which yields
	\begin{align}
	\nonumber		\frac{1}{2} \left( \delta_{t} \| \mathcal{W}_1^{n+1} \|_{0, \Omega}^2 + \Delta t \| \delta_{t}\mathcal{W}_1^{n+1} \|_{0, \Omega}^2\right) + D_1^{\min} | \mathcal{W}_1^{n+1}|^2_{1,\Omega} \leq \| f_1^{n+1}-f_2^{n+1}\|_{0,\Omega} \| \mathcal{W}_1^{n+1} \|_{0,\Omega} \nonumber  \\
	+ \frac{1}{2} \|\delta_{t}\bu_1^{s,n+1}\|_{1, \infty, \Omega} \| \mathcal{W}_1^{n+1} \|_{0,\Omega}^2 + \| w_1^{2,n+1}\|_{1,\infty, \Omega} \| \delta_{t} \mathcal{U}^{n+1}\|_{0,\Omega} \| \mathcal{W}_1^{n+1}\|_{0,\Omega}, \cblue{\label{bound1}}
	\end{align}
	and applying Cauchy-Schwarz and Young inequalities together with \eqref{eq:reaction-assumption} and \cblue{the boundedness of $\|\bu_1^{s,n+1} - \bu_1^{s,n}\|_{1, \infty, \Omega}$}, we get the bound
	\begin{align}
	\nonumber     \frac{1}{2} \left( \delta_{t} \| \mathcal{W}_1^{n+1} \|_{0, \Omega}^2 + \Delta t \| \delta_{t}\mathcal{W}_1^{n+1} \|_{0, \Omega}^2\right) + D_1^{\min} | \mathcal{W}_1^{n+1}|^2_{1,\Omega} \leq C \bigg(  \| \mathcal{W}_1^{n+1}\|_{0,\Omega}^2 + \| \mathcal{W}_2^{n+1}\|_{0,\Omega}^2 \\
	+  \frac{1}{2 \Delta t} \| \mathcal{W}_1^{n+1} \|_{0,\Omega}^2
	+  \frac{\Delta t}{2} \| \delta_{t} \mathcal{U}^{n+1}\|_{0,\Omega}^2
	+ \frac{1}{2 \Delta t} \| \mathcal{W}_1^{n+1}\|_{0,\Omega}^2 \bigg). \cblue{\label{bound2}}
	\end{align}
	Multiplying \eqref{bound2} by $\Delta t$ and taking summation over $n$, we deduce that
	\begin{align}
	\nonumber       \| \mathcal{W}_1^{n+1} \|_{0, \Omega}^2 + \Delta t^2 \sum_{m=0}^n \| \delta_{t}\mathcal{W}_1^{m+1} \|_{0, \Omega}^2  + D_1^{\min} \Delta t \sum_{m=0}^n  |\mathcal{W}_1^{m+1}|^2_{1,\Omega} \; \leq \;  C \bigg(\| \mathcal{W}_1^{0} \|_{0, \Omega}^2 + \|\mathcal{U}^{0} \|^2_{0,\Omega} \\
	+  \sum_{m=0}^n \big( (1+ \Delta t) \|\mathcal{W}_1^{m+1}\|_{0,\Omega}^2 + \| \mathcal{W}_2^{m+1} \|_{0,\Omega}^2 \big) +  \sum_{m=0}^n \|\mathcal{U}^{m+1}\|^2_{0,\Omega}
	\bigg). \cblue{\label{bound3}}
	\end{align}
	\cblue{Now, proceeding for equation \eqref{uniquew2} in a similar way as done for \eqref{bound1}-\eqref{bound3}}, we can obtain the same bound for  $\mathcal{W}_2^{n+1}$, which together with  \eqref{bound3} gives \eqref{bound-w}.
\end{proof}

\cblue{With the previous two results, we are in a position to establish the announced property of the weak solution to problem \eqref{weak-u}-\eqref{weak-w2}.}
\cblue{\begin{theorem}\label{thm:unique-h}
		The semi-discrete weak formulation  \eqref{weak-u}-\eqref{weak-w2} of the coupled problem \eqref{eq:coupled} has a unique solution.
	\end{theorem}
	\begin{proof}
			The desired estimate is established by combining  \eqref{bound-up} and \eqref{bound-w}, and Gronwall's lemma
			\begin{align*}
			\| \mathcal{U}^{n+1}\|_{1,\Omega} + \| \mathcal{P}^{n+1}\|_{0,\Omega} + \| \bigchi^{n+1}\|_{0,\Omega}
			+ \| \mathcal{W}_1^{n+1} \|_{0, \Omega} + \| \mathcal{W}_2^{n+1} \|_{0, \Omega} +  \| \mathcal{P}\|_{l^2(H^1(\Omega))} +  \| \nabla \mathcal{W}_1 \|_{l^2(L^2(\Omega))}  \\
			+ \| \nabla \mathcal{W}_2 \|_{l^2(L^2(\Omega))} \le C\biggl( \|\mathcal{U}^0 \|_{1,\Omega} + \|\mathcal{P}^0 \|_{0,\Omega} + \|\bigchi^0 \|_{0,\Omega} +	\| \mathcal{W}_1^0 \|_{0,\Omega} + \| \mathcal{W}_2^0 \|_{0,\Omega}  \\
			  + \sum_{m=0}^n \| \bb_1^{m+1} - \bb_2^{m+1}\|_{0, \Omega} + \| \ell_1 - \ell_2 \|_{l^2(L^2(\Omega))} \biggr),
			\end{align*}
		from which, we can ensure the existence of at most one weak solution to the system  \eqref{weak-u}-\eqref{weak-w2}.
	\end{proof}}
	\subsection{Continuous dependence on data}
	\begin{lemma}
		The solution $(\bu^{s,n+1},p^{f,n+1},\psi^{n+1},w_1^{n+1},w_2^{n+1})\in\bV \times Q\times Z\times H^1(\Omega)\times H^1(\Omega)$ of
		problem \eqref{weak-u}-\eqref{weak-w2}  satisfies
		\begin{align*}
		\begin{split}
		&\Vert \bu^{s,n+1}\Vert_{1,\Omega} + \sqrt{c_0}\Vert p^{f,n+1}\Vert_{0,\Omega} + \Vert \psi^{n+1}\Vert_{0,\Omega} + \Vert p^{f}\Vert_{\ell^2(H^1(\Omega))}+ \Vert w_{1}^{n+1}\Vert_{0,\Omega} + \Vert w^{n+1}_{2}\Vert_{0,\Omega}\\
		&\leq C \sqrt{\exp} \Big\{n\Delta t + \Vert\bu^{s,0}\Vert_{1,\Omega} + \Vert p^{f,0}\Vert_{0,\Omega}+ \Vert \psi^{0}\Vert_{0,\Omega}+\Vert w_{1}^{0}\Vert_{0,\Omega}+\Vert w_{2}^{0}\Vert_{0,\Omega}+\sum_{m=0}^n \Vert\bb^{m+1}\Vert_{0,\Omega} + \norm{\ell}_{\ell^2(L^2(\Omega))} \Big\}.
		\end{split}
		\end{align*}
		where $C>0$ is a constant independent of $\Delta t$ and $\lambda$.
	\end{lemma}
	\begin{proof}
		We focus first on the Biot system. Proceeding as in the proof of Lemma \ref{uncoupled-P}, we take $\bv^s=\delta_t \bu^{s,n+1}$, $q^f=p^{f,n+1}$ and $\phi=\delta_t \psi^{n+1}$ in \eqref{weak-u}, \eqref{weak-p} and \eqref{weak-psi}, respectively, to obtain
		\begin{align}
\nonumber
		&\Vert \bu^{s,n+1}\Vert^2_{1,\Omega} + c_0\Vert p^{f,n+1}\Vert^2_{0,\Omega} + \Vert \psi^{n+1}\Vert^2_{0,\Omega} +\Delta t\sum_{m=0}^{n}\Vert p^{f,m+1}\Vert^2_{1,\Omega} \\
\label{stability-result-final}		&\qquad \leq C_1\Big\{\Vert\bu^{s,0}\Vert^2_{1,\Omega} + \Vert p^{f,0}\Vert^2_{0,\Omega}+ \Vert \psi^{0}\Vert^2_{0,\Omega}+ \sum_{m=0}^{n}\Vert\psi^{m+1}\Vert_{0,\Omega}^2+\sum_{m=0}^{n}\Vert p^{f,m+1}\Vert^2_{0,\Omega}\\
\nonumber		&+\sum_{m=0}^{n}\Vert {r}^{m+1}\Vert_{0,\Omega}^2+\sum_{m=0}^{n} \Vert\bb^{m+1}\Vert_{0,\Omega}^2 + \Delta t\sum_{m=0}^{n} \norm{\ell^{m+1}}^2_{0,\Omega} \Big\} + C_2\Big\{\Vert {r}^{n+1}\Vert^2_{0,\Omega}+\Vert\bb^{n+1}\Vert^2_{0,\Omega}\Big\}.
		\end{align}
		In turn, for the ADR problem, we proceed as in the proof of Lemma \ref{uncoupled-ADR}, taking $s_1=w_{1}^{n+1}$ and $s_2=w_{2}^{n+1}$ in \eqref{weak-w1} and \eqref{weak-w2}, respectively, to get
		\begin{align}\label{bound-w1h-w2h-final}
		\begin{split}
		&\Vert w_{1}^{n+1}\Vert^2_{0,\Omega}+\Vert w_{2}^{n+1}\Vert^2_{0,\Omega}+ \Delta t\sum_{m=0}^{n}(\Vert \nabla w_{1}^{m+1}\Vert^2_{0,\Omega}+ \Vert \nabla w_{2}^{m+1}\Vert^2_{0,\Omega})\\
		& \qquad \qquad\qquad \qquad \leq C_3\Big\{n+ \Vert w_{1}^{0}\Vert^2_{0,\Omega}+\Vert w_{2}^{0}\Vert^2_{0,\Omega}+\sum_{m=0}^{n}\Big(\Vert w_{1}^{m+1}\Vert^2_{0,\Omega}+\Vert w_{2}^{m+1}\Vert^2_{0,\Omega}\Big) \Big \}.
		\end{split}
		\end{align}	
		Combining \eqref{stability-result-final} and \eqref{bound-w1h-w2h-final}, we obtain a preliminar stability bound for the coupled system \eqref{weak-u}-\eqref{weak-w2} 	
		\begin{align*}
		 &\Vert \bu^{s,n+1}\Vert^2_{1,\Omega} + c_0\Vert p^{f,n+1}\Vert^2_{0,\Omega} + \Vert \psi^{n+1}\Vert^2_{0,\Omega} +\Delta t\sum_{m=0}^{n}\Vert p^{f,m+1}\Vert^2_{1,\Omega}+\Vert w_{1}^{n+1}\Vert^2_{0,\Omega}+\Vert w_{2}^{n+1}\Vert^2_{0,\Omega} \\
		&\qquad \leq C_1\Big\{\Vert\bu^{s,0}\Vert^2_{1,\Omega} + \Vert p^{f,0}\Vert^2_{0,\Omega}+ \Vert \psi^{0}\Vert^2_{0,\Omega}+ \sum_{m=0}^{n}\Vert\psi^{m+1}\Vert_{0,\Omega}^2+\sum_{m=0}^{n}\Vert p^{f,m+1}\Vert^2_{0,\Omega}\\
		&\qquad \quad+\sum_{m=0}^{n}\Vert {r}^{m+1}\Vert_{0,\Omega}^2+\sum_{m=0}^{n} \Vert\bb^{m+1}\Vert_{0,\Omega}^2 + \Delta t\sum_{m=0}^{n} \norm{\ell^{n+1}}^2_{0,\Omega} \Big\} + C_2\Big\{\Vert {r}^{n+1}\Vert^2_{0,\Omega}+\Vert\bb^{n+1}\Vert^2_{0,\Omega}\Big\}\\
		&\qquad\quad+ C_3\Big\{n\Delta t+ \Vert w_{1}^{0}\Vert^2_{0,\Omega}+\Vert w_{2}^{0}\Vert^2_{0,\Omega}+\sum_{m=0}^{n}\Big(\Vert w_{1}^{m+1}\Vert^2_{0,\Omega}+\Vert w_{2}^{m+1}\Vert^2_{0,\Omega}\Big) \Big \},
		\end{align*}	
		and therefore, recalling the bound for $r$ given in Section \ref{preliminaries}, and applying Gronwall's inequality to the resulting estimate, we obtain the desired result.
	\end{proof}
\begin{remark}\cblue{
We have demonstrated the well-posedness of the fully coupled system by considering the time discretisation which is one of the main purpose of this contribution, and the analysis of continuous in time problem is not presented here explicitly, and which also could be of potential  interest as discussed by many researchers, for instance, see \cite{anaya18}. We stress that the analysis of continuous in time problem also can be established by proceeding analogously to the analysis presented here in the context of time discretisation and adopting the similar arguments used in \cite{anaya18} with appropriate choices of Sobolev spaces.}
\end{remark}

	\section{Mixed-primal Galerkin method}\label{sec:FE}
	\subsection{Fully discrete formulation}
	Let us consider a family $ \lbrace \cT_h \rbrace_{h>0} $
	of shape-regular, quasi-uniform partitions of the spatial domain $\bar{\Omega}$ into affine elements (triangles in 2D or
	tetrahedra in 3D) $E$ of
	diameter $h_E$, where $h = \max\{ h_E:\, E\in \cT_h\}$ denotes the mesh size.
	Finite-dimensional subspaces of the functional spaces employed in Section~\ref{sec:solvability} will be
	defined in the following manner
	\begin{align}
	\bV_h &:= \cblue{\lbrace \bv^s_h \in \mathbf{C}(\overline{\Omega}): \bv^s_h|_{E} \in [\mathbb{P}_1(E)\oplus {\rm span}\{b_E\}]^d\ \forall E\in \cT_h, \text{ and } \bv^s_h|_{\Gamma}= \cero\rbrace}, \nonumber\\
	Q_h &:= \cblue{\lbrace q^f_h \in {C}(\overline{\Omega}): q^f_h|_{E} \in \mathbb{P}_1({E})\ \forall E\in \cT_h, \text{ and } q^f_h|_{\Sigma}= 0\rbrace},
	\label{eq:FEspaces}
	\\
	Z_h &:= \lbrace \phi_h \in L^2(\Omega): \phi_h|_{E} \in \mathbb{P}_1({E})\ \forall E\in \cT_h\rbrace,\quad
	W_h := \lbrace w_h \in {C}(\overline{\Omega}): w_h|_{E} \in \mathbb{P}_1({E})\ \forall E\in \cT_h\rbrace,\nonumber
	\end{align}
	where $\mathbb{P}_k(E) $ denotes the space of polynomials of degree less than or equal than $ k $ defined locally over $ E \in \cT_h$, and $b_E:= \varphi_1\varphi_2\varphi_3$ is a $\mathbb{P}_3$ bubble function in $E$, and $\varphi_1,\,\varphi_2\,,\varphi_3$ are the barycentric coordinates of $E$. Let us recall that the pair $(\bV_h,Z_h)$ (known as the MINI element)
	is inf-sup stable (see, e.g., \cite{boffi13}).
	
	Considering reaction and coupling terms $f,g,r$ discretised implicitly,
	the fully discrete scheme associated with \eqref{eq:weak} is defined as: From initial data $\bu^{s,0},p^{f,0},\psi^{0}, w_1^0, w_2^0$
	(which will be projections of the continuous initial conditions of each field) and
	for $n=1,\ldots$, find $\bu_h^{s,n+1}\in\bV_h,p_h^{f,n+1}\in Q_h,\psi_h^{n+1}\in Z_h,w_{1,h}^{n+1}\in W_h,w_{2,h}^{n+1}\in W_h$ such that
	\begin{alignat}{5}
	&&   a_1(\bu_h^{s,n+1},\bv_h^s)   &&                 &\;+&\; b_1(\bv_h^s,\psi_h^{n+1})     &=&\;F_{r_h^{n+1}}(\bv_h^s)&\; \;\forall \bv_h^s\in\bV_h, \label{weak-u-h}\\
	\tilde{a}_2(p_h^{f,n+1},q_h^f)   &\;+&               &&      a_2(p_h^{f,n+1},q_h^f)   &\;-&\;   \tilde{b}_2(q_h^f,\psi_h^{n+1})  &=&\; G_{\ell^{n+1}}(q_h^f) &\; \; \forall q_h^f\in Q_h, \label{weak-p-h}\\
	&&b_1(\bu_h^{s,n+1},\phi_h)  &\;+\;& b_2(p_h^{f,n+1},\phi_h)&\;-&\; a_3(\psi_h^{n+1},\phi_h) &=&\; 0 &\; \; \forall\phi_h\in Z_h, \label{weak-psi-h}    \\
	\tilde{a}_4(w_{1,h}^{n+1},s_{1,h}) &\;+ &\; a_4(w_{1,h}^{n+1},s_{1,h}) &\;+ &  c(w_{1,h}^{n+1},s_{1,h},\bu_h^{s,n+1})  & & & = &\; J_{f_h^{n+1}}(s_{1,h}) &\; \; \forall s_{1,h} \in W_h, \label{weak-w1-h}\\
	\tilde{a}_5(w_{2,h}^{n+1},s_{2,h}) &\;+ &\; a_5(w_{2,h}^{n+1},s_{2,h}) &\;+ &   c(w_{2,h}^{n+1},s_{2,h},\bu_h^{s,n+1}) &&& =&\; J_{g_h^{n+1}}(s_{2,h}) &\; \; \forall s_{2,h}\in W_h. \label{weak-w2-h}
	\end{alignat}
	\subsection{Stability of the discrete solutions}
	\cblue{The following two lemmas will serve to establish the stability result for the discrete solutions.}
	\cblue{\begin{lemma}
		Assume that $(\bu_h^{s,n+1}, p_h^{f,n+1}, \psi_h^{n+1}, {w}^{n+1}_{1,h}, {w}^{n+1}_{2,h}) \in \mathbf{V}_h\times Q_h\times Z_h\times W_h\times W_h$ is solution of problem \eqref{weak-u-h}-\eqref{weak-w2-h}. Then
\begin{align}
\nonumber
&\frac{1}{2\lambda}\Vert\psi_h^{n+1}\Vert^2_{0,\Omega} + \frac{1}{2}\Big(c_0+\frac{\alpha^2}{\lambda}\Big)\Big(\Vert p^{f,n+1}_h\Vert^2_{0,\Omega} + \Delta t^2 \sum_{m=0}^{n}\Vert \delta_t p^{f,m+1}_h\Vert^2_{0,\Omega}\Big)+\frac{\kappa_1c_p\Delta t}{2\eta}\sum_{m=0}^{n}\Vert p^{f,m+1}_h\Vert^2_{1,\Omega}\\
\label{preliminar-stab-p-psi-4-dis}		& \leq \frac{1}{2\lambda}\Vert\psi_h^{0}\Vert^2_{0,\Omega}+ \frac{1}{2}\Big(c_0+\frac{\alpha^2}{\lambda}\Big)\Vert p_h^{f,0}\Vert^2_{0,\Omega}+ \frac{2\alpha^2}{\lambda}\sum_{m=0}^{n}\Vert p^{f,m+1}_h\Vert^2_{0,\Omega}+\frac{\eta\Delta t}{2\kappa_1c_p}\sum_{m=0}^{n} \norm{\ell^{m+1}}^2_{0,\Omega}\\
\nonumber		&\, + \frac{1}{\mu C_{k,1}}\Vert\psi_h^{n+1}\Vert^2_{0,\Omega}+ \frac{\mu C_{k,1}}{2}\Vert\bu_h^{s,n+1}\Vert^2_{1,\Omega}+\frac{2}{\mu C_{k,1}}\sum_{m=0}^{n-1}\Vert\psi_h^{m+1}\Vert^2_{0,\Omega}+\frac{\mu C_{k,1}\Delta t^2}{4}\sum_{m=0}^{n-1}\Vert\delta_t\bu_h^{m+1}\Vert^2_{1,\Omega},
\end{align}
\begin{align}\label{stability-uh-final-dis}
\begin{split}
&\mu C_{k,1}\Vert \bu_h^{s,n+1}\Vert^2_{1,\Omega} + \frac{\mu C_{k,1}\Delta t^2}{4}\sum_{m=0}^{n}\Vert\delta_t \bu_h^{s,m+1}\Vert^2_{1,\Omega} \\
&\qquad\qquad  \leq  C_1\Big\{\Vert\bu_h^{s,0}\Vert^2_{1,\Omega} + \sum_{m=0}^{n}\Vert\psi_h^{m+1}\Vert_{0,\Omega}^2+\sum_{m=0}^{n}\Vert {r}^{m+1}_h\Vert_{0,\Omega}^2+\sum_{m=0}^{n} \Vert\bb^{m+1}\Vert_{0,\Omega}^2\Big\},
\end{split}
\end{align}
and
\begin{align}\label{inf-sup-psi-dis}
\begin{split}
&\Vert\psi^{n+1}_h\Vert^2_{0,\Omega}\leq C_2\Big\{\Vert\bu_h^{s,n+1}\Vert^2_{1,\Omega}+ \Vert {r}^{n+1}_h\Vert^2_{0,\Omega}+\Vert\bb^{n+1}\Vert^2_{0,\Omega}\Big\},
\end{split}
\end{align}
where $C_1, C_2$ are positive constants independent of $\lambda, h,$ and $\Delta t$.
	\end{lemma}}
\begin{proof}
We proceed similarly to the proof of Lemmas \ref{uncoupled-P} and \ref{uncoupled-ADR}. We focus first on the stability of  \eqref{weak-u-h}-\eqref{weak-psi-h}. Taking $\bv^s_h=\delta_t\bu_h^{s,n+1}$ in \eqref{weak-u-h}, using Cauchy-Schwarz inequality, applying Young's inequality with constants chosen conveniently, and then, summing over $n$ and multiplying by $\Delta t$, we readily get \eqref{stability-uh-final-dis},
where $C_1$ is a constant depending on $\mu, C_{k,1},C_{k,2}, \rho$, and $ \tau$.
Now, in equations \eqref{weak-p-h} and \eqref{weak-psi-h}, we take \cblue{$q^f_h=p^{f,n+1}_h$} and $\phi_h=\delta_t\psi^{n+1}_h$, respectively, to obtain
\begin{align}
\nonumber
&\frac{1}{2\lambda}\delta_t\Vert\psi_h^{n+1}\Vert^2_{0,\Omega} + \frac{\Delta t}{2\lambda}\Vert\delta_t\psi_h^{n+1}\Vert^2_{0,\Omega} + \frac{1}{2}\Big(c_0+\frac{\alpha^2}{\lambda}\Big)\Big(\delta_t\Vert p^{f,n+1}_h\Vert_{0,\Omega} + \Delta t \Vert \delta_t p^{f,n+1}_h\Vert_{0,\Omega}\Big)+\frac{\kappa_1}{\eta}|p^{f,n+1}_h|^2_{1,\Omega}\\
& \qquad\qquad \qquad \leq \frac{2\alpha}{\lambda}\Vert p^{f,n+1}_h\Vert_{0,\Omega}\Vert \delta_t \psi^{n+1}_h\Vert_{0,\Omega} + \norm{\ell^{n+1}}_{0,\Omega}\Vert p^{f,n+1}_h\Vert_{0,\Omega}- \int_{\Omega}\delta_t \psi^{n+1}_h \mathrm{div}\,\bu^{s,n+1}_h,
\label{preliminar-stab-p-psi-dis}
\end{align}
Thus, applying Young's inequality to the first and second term, and summation by parts to the last term, on the right-hand side of \eqref{preliminar-stab-p-psi-dis}, we obtain \eqref{preliminar-stab-p-psi-4-dis}.

On the other hand, as in Lemma \ref{uncoupled-P} we target an estimate independent of $\lambda$. For that reason we use the discrete version of the inf-sup condition \eqref{inf-sup}, which is satisfied by the finite element family \eqref{eq:FEspaces} \cite{girault79,boffi13}. Thus, taking $\phi_h=\psi_h^{n+1}$, using \eqref{weak-u-h} and the continuity of $a_1$, we obtain
\begin{align*}
\begin{split}
&\hat{\beta}\Vert\psi^{n+1}_h\Vert_{0,\Omega}\leq \sup_{\bv_h^s \in \mathbf{V}_h} \frac{b_1(\bv_h^s,\psi_h^{n+1})}{\|\bv_h^s\|_{1,\Omega}}  = \sup_{\bv_h^s \in \mathbf{V}_h} \frac{-a_1(\bu^{s,n+1}_h,\bv_h^{s})+F_{r_h^{n+1}}(\bv^{s}_h)}{\|\bv_h^s\|_{1,\Omega}}\\
&\qquad  \leq 2\mu C_{k,2} \Vert\boldsymbol{\varepsilon}(\bu_h^{n+1})\Vert_{0,\Omega}+\sqrt{C_{k,2}}\tau \Vert r^{n+1}_h\Vert_{0,\Omega}+\rho\Vert\bb^{n+1}\Vert_{0,\Omega},
\end{split}
\end{align*}
which can be written equivalently as \eqref{inf-sup-psi-dis},
with $C_2$ depending on $C_{k,1},C_{k,2}, \mu, \tau, \rho$ and the discrete inf-sup constant $\hat{\beta}$.
\end{proof}

\cblue{\begin{lemma}
	Assume that $(\bu_h^{s,n+1}, p_h^{f,n+1}, \psi_h^{n+1}, {w}^{n+1}_{1,h}, {w}^{n+1}_{2,h}) \in \mathbf{V}_h\times Q_h\times Z_h\times W_h\times W_h$ is solution of problem \eqref{weak-u-h}-\eqref{weak-w2-h}. Then
	\begin{align}\label{bound-wh-dis}
	\begin{split}
	&\frac{1}{2}\sum_{i=1}^{2}\Vert w_{i,h}^{n+1}\Vert^2_{0,\Omega}+ \frac{1}{2}\Delta t^2\sum_{i=1}^{2}\sum_{m=0}^{n} \Vert \delta_t w_{i,h}^{m+1}\Vert^2_{0,\Omega}+ \Delta t \sum_{i=1}^{2}\sum_{m=0}^{n}D_i^{\min}\Vert  \nabla w_{i,h}^{m+1}\Vert^2_{0,\Omega}\\
	& \qquad \leq \frac{1}{2}\sum_{i=1}^{2}\Vert w_{i,h}^{0}\Vert^2_{0,\Omega}+\frac{1}{2}(M_1+\Delta t )\sum_{i=1}^{2}\sum_{m=0}^{n}\Vert w_{i,h}^{m+1}\Vert^2_{0,\Omega}+ \frac{\Delta t}{2}\sum_{m=0}^{n}\Vert f_h^{m+1}\Vert^2_{0,\Omega}+ \frac{\Delta t}{2}\sum_{m=0}^{n}\Vert g_h^{m+1}\Vert^2_{0,\Omega}.
	\end{split}
	\end{align}
\end{lemma}}
	\begin{proof}
		Notice that for the ADR problem \eqref{weak-w1-h}-\eqref{weak-w2-h}, by
		taking $s_{1,h}=w^{n+1}_{1,h}$ in  \eqref{weak-w1-h}, we get
		\begin{align*}
		\int_\Omega  \delta_t  w_{1,h}^{n+1}w_{1,h}^{n+1} + \int_\Omega D_1(\bx)\nabla w_{1,h}^{n+1}\cdot \nabla w_{1,h}^{n+1}+  \int_\Omega  (\delta_t\bu_h^{s,n+1}\cdot\nabla w_{1,h}^{n+1})w_{1,h}^{n+1}=\int_{\Omega}f_h^{n+1}w_{1,h}^{n+1},
		\end{align*}
		and then, applying \eqref{property-term} and Cauchy-Schwarz inequality, we deduce the estimate
		\begin{align*}
		\begin{split}
		&\frac{1}{2}\delta_t\Vert w_{1,h}^{n+1}\Vert^2_{0,\Omega}+ \frac{1}{2}\Delta t \Vert \delta_t w_{1,h}^{n+1}\Vert^2_{0,\Omega}+  D_1^{\min}\Vert \nabla w_{1,h}^{n+1}\Vert^2_{0,\Omega}\\
		&\qquad\qquad \qquad\qquad\qquad\leq \frac{1}{2}\Vert\delta_t \bu_h^{s,n+1}\Vert_{1,\infty,\Omega}\Vert w_{1,h}^{n+1}\Vert^2_{0,\Omega} + \Vert f_h^{n+1}\Vert_{0,\Omega}\Vert w_{1,h}^{n+1}\Vert_{0,\Omega}.
		\end{split}
		\end{align*}
		Since $\Omega$ is a bounded domain and the elements of $\bV_h$ are piecewise polynomials, we know that $\Vert \bu^{s,n+1}_{h}-\bu^{s,n}_{h}\Vert_{1,\infty,\Omega}< +\infty$ for each $\bu^{n+1}_{h},\bu^{n}_{h}\in \bV_h$ (see, e.g., \cite{ggr-2018}), and then, without loss of generality, we may assume that $\Vert \bu^{s,n+1}_{h}-\bu^{s,n}_{h}\Vert_{1,\infty,\Omega}\leq M_1$ for some $M_1\in \mathbb{R}$. Thus, applying Young's inequality, summing over $n$ and multiplying by $\Delta t$, we obtain the following result
		\begin{align}\label{bound-w1h-dis}
		\begin{split}
		&\frac{1}{2}\Vert w_{1,h}^{n+1}\Vert^2_{0,\Omega}+ \frac{1}{2}\Delta t^2\sum_{m=0}^{n} \Vert \delta_t w_{1,h}^{m+1}\Vert^2_{0,\Omega}+  D_1^{\min}\Delta t \sum_{m=0}^{n}\Vert \nabla w_{1,h}^{m+1}\Vert^2_{0,\Omega} \\
		&\qquad \qquad \leq \frac{1}{2}\Vert w_{1,h}^{0}\Vert^2_{0,\Omega}+\frac{1}{2}(M_1+\Delta t )\sum_{m=0}^{n}\Vert w_{1,h}^{m+1}\Vert^2_{0,\Omega}+ \frac{\Delta t}{2}\sum_{m=0}^{n}\Vert f_h^{m+1}\Vert^2_{0,\Omega}.
		\end{split}
		\end{align}
		Moreover, we realise that an estimate for $\Vert w_{2,h}^{n+1}\Vert_{0,\Omega}$ stays exactly as above, which is
		\begin{align}\label{bound-w2h-dis}
		\begin{split}
		&\frac{1}{2}\Vert w_{2,h}^{n+1}\Vert^2_{0,\Omega}+ \frac{1}{2}\Delta t^2\sum_{m=0}^{n} \Vert \delta_t w_{2,h}^{m+1}\Vert^2_{0,\Omega}+ D_2^{\min}\Delta t \sum_{m=0}^{n}\Vert  \nabla w_{2,h}^{m+1}\Vert^2_{0,\Omega}\\
		&\qquad \qquad \leq \frac{1}{2}\Vert w_{2,h}^{0}\Vert^2_{0,\Omega}+\frac{1}{2}(M_1+\Delta t )\sum_{m=0}^{n}\Vert w_{2,h}^{m+1}\Vert^2_{0,\Omega}+ \frac{\Delta t}{2}\sum_{m=0}^{n}\Vert g_h^{m+1}\Vert^2_{0,\Omega},
		\end{split}
		\end{align}
		therefore completing the proof.
	\end{proof}
	
\cblue{Finally, we can establish the stability result for the discrete solution.}
\cblue{	\begin{lemma}
		Assume that $(\bu_h^{s,n+1}, p_h^{f,n+1}, \psi_h^{n+1}, {w}^{n+1}_{1,h}, {w}^{n+1}_{2,h}) \in \mathbf{V}_h\times Q_h\times Z_h\times W_h\times W_h$ is solution of problem \eqref{weak-u-h}-\eqref{weak-w2-h}. Then, there exists $C>0$ independent of $\lambda, h$, and $\Delta t$, such that
		\begin{align}\label{stability-fully-discrete}
		\begin{split}
		&\Vert \bu_h^{s,n+1}\Vert_{1,\Omega} + \sqrt{c_0}\Vert p^{f,n+1}_h\Vert_{0,\Omega} + \Vert \psi^{n+1}_h\Vert_{0,\Omega} + \Vert p^{f}_h\Vert_{\ell^2(H^1(\Omega))}+ \Vert w_{1,h}^{n+1}\Vert_{0,\Omega} + \Vert w^{n+1}_{2,h}\Vert_{0,\Omega}\\
		&\leq C \sqrt{\exp} \Big\{n\Delta t+ \Vert\bu_h^{s,0}\Vert_{1,\Omega} + \Vert p_h^{f,0}\Vert_{0,\Omega}+ \Vert \psi_h^{0}\Vert_{0,\Omega}+\Vert w_{1,h}^{0}\Vert_{0,\Omega}\\
		&\qquad \qquad \qquad +\Vert w_{2,h}^{0}\Vert_{0,\Omega}+  \sum_{m=0}^{n+1}\Vert\bb^{m+1}\Vert_{0,\Omega} + \norm{\ell}_{\ell^2(L^2(\Omega))} \Big\}.
		\end{split}
		\end{align}
	\end{lemma}
	\begin{proof}
		The result \eqref{stability-fully-discrete} follows from the growth condition on $f_h$ and $g_h$, adding \eqref{stability-uh-final-dis}, \eqref{preliminar-stab-p-psi-4-dis}, \eqref{inf-sup-psi-dis}, and \eqref{bound-wh-dis}, recalling the bound for $r$, and applying the discrete Gronwall's inequality.
	\end{proof}}
\begin{remark}
			The solvability analysis of \eqref{weak-u-h}-\eqref{weak-w2-h} can be established similarly to the continuous case. More precisely, as in Section \ref{fixed-point-operator} we need to define a fixed-point operator, whose well-definiteness will depend upon the solvability of each uncoupled problem. For the discrete poroelasticity system we can adapt the  analysis from \cite[Section 3]{oyarzua16}, whereas for the approximate ADR equations we can apply classical techniques for discrete quasi-linear problems \cite{quarteroni94}. Next, we need to prove the continuity of the operator going from $[W_h]^2$ into itself, which follows as a consequence of the estimate \eqref{stability-fully-discrete} in combination with the ideas employed in \cite[Section 5.3]{anaya18}. Finally, the result follows from an application of the well-known Brouwer fixed-point theorem.
	\end{remark}
	
\begin{remark}	\cblue{
We stress that the all the arguments and techniques used in proving the stability of the discrete-in-time problem, may not be directly applicable for ensuring the stability of the proposed fully discrete scheme, as the discrete variables  involved in the formulation may  not have enough regularity as demanded in the semi-discrete analysis. Moreover, the ideas developed in illustrating the stability of a fully discrete scheme will be repeatedly used in the establishment of error estimates.}
\end{remark}

\cblue{\section{Error estimates}\label{sec:error}
In order to see the rate of convergence of the proposed fully discrete scheme, we will derive the error estimates in suitable norms for each of the variables that appear in the formulation. For establishing the error estimates, we will be utilising the well-known techniques/arguments used for time-dependent problems and imitating the steps used in showing stability. Therefore, we would like to provide a brief sketch of the proof by citing the appropriate references for more details. First, we define the following projection operator
  $$\mathbf{A}_h:=(A_h^{\bu}, A_h^p, A_h^{\psi}, A_h^{w_1}, A_h^{w_2}),$$
  where   ($A_h^{\bu}, A_h^{\psi}$) and  $A_h^p, A_h^{w_1}, A_h^{w_2}$ are standard Stokes operator and elliptic projections respectively,  defined as follows, $\forall  \bv_h \in \bV_h, \phi_h \in Z_h, \forall q_h \in Q_h$ and $\forall w_i \in W_h, i=1,2$,
  \begin{align}
a_1(A_h^u \bu, \bv_h) + b_1(\bv_h, A_h^{\psi} \psi) & = a_1(\bu, \bv_h)+ b_1(\bv_h, \psi);\quad  b_1(A_h^u \bu, \phi_h)  = b_1(\bu, \phi_h); \label{eq:proj-u-psi}\\
a_2(A_h^p p, q_h) & =a_2(p,q_h); \quad
(\nabla A_h^{w_i} w_i, \nabla s_{i,h}) = (\nabla w_i, \nabla s_{i,h}). \label{eq:proj-p-w}
\end{align}
  These operators satisfy the following estimates (see, for instance, \cite{girault79,quarteroni94}):
	\begin{align}
	\| \bu - A_h^{\bu} \bu \|_{0,\Omega} + h ( |\bu - A_h^{\bu} \bu|_{1,\Omega} + \| \psi - A_h^{\psi} \psi \|_{0,\Omega} )  & \le C h^2, \label{est-Au-Apsi} \\
	\| p - A_h^{p} p \|_{0,\Omega} + h |p - A_h^{p} p|_{1,\Omega} & \le C h^2, \label{est-Ap}\\
	\| w_i - A_h^{w_i} w_i\|_{0,\Omega} + h |w_i - A_h^{w_i} w_i|_{1,\Omega} & \le C h^2, \quad i=1,2. \label{est-Aw}
	\end{align}
\begin{theorem}\label{theo:error}
		Let  $(\bu(t), p(t), \psi(t), w_1(t), w_2(t))$  and $(\bu_h^{n+1}, p_h^{n+1}, \psi_h^{n+1}, w_{1,h}^{n+1}, w_{2,h}^{n+1})$ be the unique solutions to the systems \eqref{eq:weak} and \eqref{weak-u-h}-\eqref{weak-w2-h}, respectively. Then the following estimate holds, with constant $C$ independent of $h$ and $\Delta t$,
	\begin{align}
	\Vert \bu^{n+1} -\bu_h^{n+1} \Vert_{1,\Omega}^2 + \Vert \psi^{n+1} -\psi_h^{n+1} \Vert_{0,\Omega}^2  + (\Delta t) \sum_{k=0}^{n} | p^{k+1} -p_h^{k+1} |_{1,\Omega}^2  & \nonumber \\
	+ (\Delta t) \sum_{k=0}^{n} \Big( | w_1^{k+1} -w_{1,h}^{k+1} |_{1,\Omega}^2 + |  w_2^{k+1} -w_{2,h}^{k+1} |_{1,\Omega}^2 \Big)& \le C (h^2 + \Delta t^2). \label{est-fully-discrete}
	\end{align}
\end{theorem}	
\begin{proof}
First we decompose the error as follows for each $t$ and $i=1,2$:
\begin{align*}
\xi-\xi_h & =\xi-\mathbf{A}_h+\mathbf{A}_h-\xi_h,\\
& =(\underbrace{\bu-A_h^{\bu}}_{:=\rho_{\bu}}+\underbrace{A_h^{\bu}-\bu_h}_{:=\eta_{\bu}}, \underbrace{p-A_h^{p}}_{:\rho_p}
+\underbrace{A_h^{p}-p_h}_{:=\eta_p}, \underbrace{\psi-A_h^{\psi}}_{:\rho_{\psi}}+\underbrace{A_h^{\psi}-\psi_h}_{:=\eta_{\psi}}, \underbrace{w_i-A_h^{w_i}}_{:=\rho_{w_i}}
+\underbrace{A_h^{w_i}-w_{i,h}}_{:=\eta_{w_i}}),\end{align*}
where $\xi=(\bu, p, \psi, w_1,  w_2)$ and $\xi_h=(\bu_h, p_h, \psi_h, w_{1,h}, w_{2,h})$.
On subtracting \eqref{weak-u-h}-\eqref{weak-w2-h} from \eqref{eq:weak}, choosing $\bv_h = \delta_{t} \eta_{\bu}^{n+1}, \phi_h = \eta_{\psi}^{n+1}, q_h = \eta_p^{n+1}$, $s_{1,h}= \eta_{w_1}^{n+1}$ and $s_{2,h}= \eta_{w_2}^{n+1}$ and invoking \eqref{eq:proj-u-psi} and \eqref{eq:proj-p-w}, enable us to write the following error equations
\begin{align}
	a_1(\eta_{\bu}^{n+1}, \delta_{t} \eta_{\bu}^{n+1}) + b_1(\delta_{t} \eta_{\bu}^{n+1}, \eta_{\psi}^{n+1}) & = (F_{r^{n+1}}- F_{r_h^{n+1}})(\delta_{t} \eta_{\bu}^{n+1}), \label{eq:err-u}\\
	\tilde{a}_2( \eta_p^{n+1}, \eta_p^{n+1}) + a_2(\eta_p^{n+1}, \eta_p^{n+1}) - \tilde{b}_2(\eta_p^{n+1}, \eta_{\psi}^{n+1}) &
	= - \tilde{a}_2(\rho^{n+1}_p, \eta_p^{n+1}) + \tilde{b}_2(\eta_p^{n+1}, \rho_{\psi}^{n+1}) \nonumber \\
	& \quad - \Big(c_0 + \frac{\alpha^2}{\lambda} \Big) (\partial_t p(\cdot, t_{n+1}) - \delta_{t} p^{n+1}, \eta_p^{n+1})  \label{eq:err-p} \\
	& \quad - \Big( \frac{\alpha}{\lambda} \Big) (\eta_p^{n+1}, \partial_t \psi - \delta_t \psi^{n+1}), \nonumber \\
	b_1(\eta_{\bu}^{n+1}, \eta_{\psi}^{n+1}) +b_2(\eta_p^{n+1}, \eta_{\psi}^{n+1}) -a_3(\eta_{\psi}^{n+1}, \eta_{\psi}^{n+1}) & =  - b_2(\rho_p^{n+1}, \eta_{\psi}^{n+1}) + a_3(\rho_{\psi}^{n+1}, \eta_{\psi}^{n+1}), \label{eq:err-psi}\\
	\tilde{a}_4(\eta_{w_1}^{n+1}, \eta_{w_1}^{n+1}) + a_4(\eta_{w_1}^{n+1}, \eta_{w_1}^{n+1}) & = J_{f^{n+1} -  f_h^{n+1}}(\eta_{w_1}^{n+1}) - \tilde{a}_4(\rho_{w_1}^{n+1}, \eta_{w_1}^{n+1})  \nonumber \\
	& \quad - (\partial_t w_1(\cdot, t_{n+1}) - \delta_{t} w_1^{n+1}, \eta_{w_1}^{n+1}) \label{eq:err-w1} \\
	& \quad - \Big(c(w_1^{n+1}, \eta_{w_1}^{n+1}, \bu^{s,n+1}) - c(w_{1,h}^{n+1}, \eta_{w_1}^{n+1}, \bu_h^{s,n+1}) \Big), \nonumber \\
	\tilde{a}_5(\eta_{w_2}^{n+1}, \eta_{w_2}^{n+1}) + a_5(\eta_{w_2}^{n+1}, \eta_{w_2}^{n+1}) & = J_{g^{n+1} -  g_h^{n+1}}(\eta_{w_2}^{n+1}) - \tilde{a}_5(\rho_{w_2}^{n+1}, \eta_{w_2}^{n+1}) \nonumber \\
	& \quad - (\partial_t w_2(\cdot, t_{n+1}) - \delta_{t} w_2^{n+1}, \eta_{w_2}^{n+1})
	\label{eq:err-w2}  \\
	& \quad - \Big(c(w_2^{n+1}, \eta_{w_2}^{n+1}, \bu^{s,n+1}) - c(w_{2,h}^{n+1}, \eta_{w_2}^{n+1}, \bu_h^{s,n+1}) \Big). \nonumber
	\end{align}
We then proceed to rewrite equation \eqref{eq:err-psi} for $n+1$ and $n$ and then subtracting these equations (as done in, e.g., \cite[Lemma 4.1]{lee17}). Then we combine equations \eqref{eq:err-u}-\eqref{eq:err-psi} (see also \cite{yi17}), and we then multiply by $\Delta t$ the resulting  expression together with the error equations \eqref{eq:err-w1}-\eqref{eq:err-w2}.
 Summing the result over each $n$ and proceeding similarly as in the proofs of Lemmas \ref{uncoupled-P} and \ref{uncoupled-ADR}, we arrive at
\begin{align}
&	\nonumber	\mu C_{k,1} \| \eta_{\bu}^{n+1} \|_{1,\Omega}^2 + \| \eta_{\psi}^{n+1} \|_{0,\Omega}^2 + c_0 \| \eta_p^{n+1} \|_{0,\Omega}^2 + \frac{\kappa_1}{\eta} (\Delta t) \sum_{k=0}^{n} |\eta_p^{k+1}|_{1,\Omega}^2 \\
&\nonumber \ \le \mu C_{k,1} \| \eta_{\bu}^0 \|_{1,\Omega}^2 + \Big( c_0 + \frac{\alpha^2}{\lambda} \Big) \| \eta_p^0 \|_{0,\Omega}^2+ \frac{1}{\lambda} \sum_{k=0}^{n} \| \eta_{\psi}^k \|_{0,\Omega}^2  \\
&\quad 	\nonumber	+
	\Delta t \sum_{k=0}^n \Big( (F_{r^{k+1}}- F_{r_h^{k+1}})(\delta_{t} \eta_{\bu}^{k+1}) - a_1(\rho_{\bu}^{k+1}, \delta_{t} \eta_{\bu}^{k+1}) - b_1(\delta_{t} \eta_{\bu}^{k+1}, \rho_{\psi}^{k+1})- \tilde{a}_2(\rho^{k+1}_p, \eta_p^{k+1}) \\
&\qquad \qquad \ \quad
	- a_2(\rho_p^{k+1}, \eta_p^{k+1}) + \tilde{b}_2(\eta_p^{k+1}, \rho_{\psi}^{k+1}) +b_1( \delta_{t} \rho_{\bu}^{k+1}, \eta_{\psi}^{k+1}) - b_2( \delta_{t} \rho_p^{k+1}, \eta_{\psi}^{k+1})\label{ineq:u-p}\\
&\qquad \qquad \ \quad	 + a_3(\delta_{t} \rho_{\psi}^{k+1}, \eta_{\psi}^{k+1}) - \Big(c_0 + \frac{\alpha^2}{\lambda} \Big) (\partial_t p(\cdot, t_{k+1}) - \delta_{t} p^{k+1}, \eta_p^{k+1})\nonumber	 \\
&\qquad \qquad \ \quad		\nonumber	- \Big( \frac{\alpha}{\lambda} \Big) (\eta_p^{k+1}, \partial_t \psi - \delta_t \psi^{k+1}) 	- (\partial_t w_1(\cdot, t_{k+1}) - \delta_{t} w_1^{k+1}, \eta_{w_1}^{k+1}) \\
&\qquad \qquad \ \quad		\nonumber- (\partial_t w_2(\cdot, t_{k+1}) - \delta_{t} w_2^{k+1}, \eta_{w_2}^{k+1}) \Big),
	\end{align}
	and
\begin{align}
&	\nonumber	\|\eta_{w_i}^{n+1} \|_{0,\Omega}^2 + D_i^{\min} (\Delta t) \sum_{k=0}^n |\eta_{w_i}^{k+1}|_{1,\Omega}^2 \\
& \quad \le \|\eta_{w_i}^{0} \|_{0,\Omega}^2 + |\eta_{w_i}^{0}|_{1,\Omega}^2 + \Delta t \sum_{k=0}^n \Big( J_{f^{k+1} -  f_h^{k+1}}(\eta_{w_i}^{k+1}) - \tilde{a}_4(\rho_{w_i}^{k+1}, \eta_{w_i}^{k+1}) - c(w_i^{k+1}, \eta_{w_i}^{n+1}, \bu^{s,k+1})  \label{ineq-w1}\\
	& \qquad \qquad\qquad \qquad\qquad \qquad \quad - c(w_{i,h}^{k+1}, \eta_{w_i}^{n+1}, \bu_h^{s,k+1}) - (\partial_t w_i(\cdot, t_{k+1}) - \delta_{t} w_i^{k+1}, \eta_{w_i}^{k+1}) \Big). \nonumber
	\end{align}
In view of  \eqref{property-of-r}, \eqref{eq:reaction-assumption}, Cauchy-Schwarz, Poincare and Young's inequalities, we obtain the following bounds for the nonlinear terms appearing in \eqref{ineq:u-p}, \eqref{ineq-w1}
\begin{align*}
\Delta t \sum_{k=0}^n (F_{r^{k+1}}- F_{r_h^{k+1}})(\delta_{t} \eta_{\bu}^{k+1})
	& \le C \Delta t \Big( \| \eta_u^0\|_{0,\Omega}^2 +  \sum_{k=0}^n (\sum_{i=1}^2 (\| \rho_{w_i}^{k+1}\|_{0,\Omega}^2 + \| \eta_{w_i}^{k+1}\|_{0,\Omega}^2) + \| \eta_u^{k+1} \|_{0,\Omega}^2 ) \Big),\label{ineq-1}\\
\Delta t \sum_{k=0}^n  J_{f^{k+1} -  f_h^{k+1}}(\eta_{w_1}^{k+1})& \le  C \Delta t \sum_{k=0}^n \Big( \| \eta_{w_1}^{k+1} \|_{0,\Omega}^2 + \| \eta_{w_2}^{k+1} \|_{0,\Omega}^2 \Big),\\
\Delta t \sum_{k=0}^n J_{g^{k+1} -  g_h^{k+1}}(\eta_{w_2}^{k+1})&\le C \Delta t \sum_{k=0}^n \Big( \| \eta_{w_2}^{k+1} \|_{0,\Omega}^2 + \| \eta_{w_2}^{k+1} \|_{0,\Omega}^2 \Big).
	\end{align*}
Then, a repeated application of Cauchy-Schwarz and Young's inequalities together with the assumption $ \| \bu^{s,n+1}\|_{1,\infty, \Omega}$ and noting that $ \| w_{i,h}^{n+1} \|_{1,\infty, \Omega}\leq C,~i=1,2$ (follow the argument similar to obtain \eqref{bound-w1h-dis}) help us in obtaining the following bound for the coupling term of \eqref{ineq-w1} for $i=1,2$.
\begin{align*}
& \Delta t \sum_{k=0}^{n}\Big( c(w_i^{k+1}, \eta_{w_i}^{k+1}, \bu^{s,k+1}) - c(w_{i,h}^{k+1}, \eta_{w_i}^{k+1}, \bu_h^{s,k+1})\Big) \\
& \quad = \Delta t \sum_{k=0}^{n} \Big( c(\rho_{w_i}^{k+1} + \eta_{w_i}^{k+1}, \eta_{w_i}^{k+1}, \bu^{s,k+1}) + c(w_{i,h}^{k+1}, \eta_{w_i}^{k+1}, \rho_{\bu}^{k+1} + \eta_{\bu}^{k+1}) \Big)\\
& \quad \le C \Delta t \sum_{k=0}^{n} \Big( | \rho_{w_1}^{k+1}|_1 \|\eta_{w_1}^{k+1} \|_0 \| \delta_t \bu^{s,k+1} \|_{0,\infty, \Omega} + \|\eta_{w_1}^{k+1} \|_0^2 \| \delta_{t} \bu^{s,k+1} \|_{1,\infty, \Omega} \\
& \qquad \quad \qquad + \|w_{1,h}^{k+1}\|_{1,\infty, \Omega} \| \eta_{w_1}^{k+1} \|_0 (\| \delta_t \eta_{\bu}^{k+1} \|_0 + \| \delta_t \rho_{\bu}^{k+1} \|_0)\Big) \\
	& \quad \le C \Big(  \| \eta_{\bu}^{0} \|_{0,\Omega}^2 + \| \rho_{\bu}^{0} \|_{0,\Omega}^2 + \sum_{k=0}^{n} \Big(\| \bu^{s,k+1} - \bu^{s,k} \|_{1,\infty, \Omega} (|\rho_{w_i}^{k+1}|_{1,\Omega}^2+ \|\eta_{w_i}^{k+1} \|_{0,\Omega}^2) \\
	&\qquad \qquad \qquad \qquad \qquad \qquad \quad+ \|w_{i,h}^{k+1}\|_{1,\infty, \Omega} ( \| \eta_{w_i}^{k+1} \|_{0,\Omega}^2 + \| \eta_{\bu}^{k+1} \|_{0,\Omega}^2 + \| \rho_{\bu}^{k+1} \|_{0,\Omega}^2) \Big) \Big).
\end{align*}
We then proceed to collect all these bounds, and we employ a proper choice of $(\bu_h^0, p_h^0, \psi_h^0, w_{1,h}^0, w_{2,h}^0)$. Next we gather these results and use Taylor's expansion in the following form: for any smooth enough function $\xi$, we have
	\begin{equation*}(\xi^{n+1} - \xi^n) - (\Delta t) \partial_t \xi(\cdot, t_{n+1})=\int_{t_n}^{t_{n+1}}(s-t_n) \partial_{tt} \xi(\cdot, s)~\ds.
\end{equation*}
We can then apply the result \ref{stability-fully-discrete} and Gronwall's inequality, which yields
\begin{align*}
	\Vert \eta_{\bu}^{n+1} \Vert_{1,\Omega}^2  + \Vert \eta_{\psi}^{n+1} \Vert_{0,\Omega}^2 + (\Delta t) \sum_{k=0}^{n} \Big( |\eta_p^{k+1}|_{1,\Omega}^2 + |\eta_{w_1}^{k+1}|_{1,\Omega}^2 + |\eta_{w_2}^{k+1}|_{1,\Omega}^2 \Big)\le C (h^2 + \Delta t^2).
	\end{align*}
Finally, the estimates \eqref{est-Au-Apsi}-\eqref{est-Aw} together with a direct application of triangle's inequality complete the rest of the proof.
\end{proof}}
	\section{Numerical tests}\label{sec:results}
 \subsection{Example 1: verification of spatio-temporal convergence}
	We have not derived theoretically error bounds, but proceed in this Section to examine numerically
	the rates of convergence of the mixed-primal scheme. Let us consider $\Omega=(0,1)^2$ with $\Gamma=\{\bx: x_1 = 0\text{ or } x_2 = 0\}$ (the bottom and left edges of the boundary) and $\Sigma=\{\bx: x_1=1 \text{ or } x_2=1\}$ (top and right sides of the square domain).
	Following \cite{kumar19}, we define
	closed-form solutions to the coupled poro-mechano-chemical system \eqref{eq:coupled} as
	\begin{equation}\label{eq:sol-ex}
	\begin{split}
	\bu^s = u_{\infty}\frac{t^2}{2}\begin{pmatrix}
	\sin(\pi x_1)\cos(\pi x_2)+\frac{x_1^2}{\lambda}\\
	-\cos(\pi x_1)\sin(\pi x_2)+\frac{x_2^2}{\lambda}
	\end{pmatrix}, \quad
	p^f = t(x_1^3 -x_2^4), \quad  \psi = p^f-\lambda\vdiv\bu^s,  \\
	w_1 = t[\exp(x_1)+\cos(\pi x_1)\cos(\pi x_2)], \quad w_2 = t[\exp(-x_2)+\sin(\pi x_1)\sin(\pi x_2)],
	\end{split}\end{equation}
	and we use these smooth functions to construct expressions for the body force $\bb(\bx,t)$, the
	fluid source $\ell(\bx,t)$, additional mass sources $S_1(\bx,t), S_2(\bx,t)$ for \eqref{eq:ADR1}-\eqref{eq:ADR2};
	a non-homogeneous displacement and non-homogeneous fluid normal flux on $\Gamma$,
	as well as non-homogeneous Dirichlet boundary pressure and non-homogeneous traction defined on $\Sigma$.
	The model parameters take the values: $u_{\infty} = \alpha=\gamma=0.1$,
	$c_0=\eta=10^{-3}$, $\kappa= 10^{-4}$, $D_1=0.05$, $D_2= \rho$,
	$\beta_1 = 170$, $\beta_2=0.1305$, $\beta_3=0.7695$, $\mu =10033.444$,
	$\lambda=993311.037$, and $\tau =10^5$. For this example we simply take the function that modulates the active stress in \eqref{eq:active-stress}
	as $r=w_1+w_2$ and use $\bk = (1,0)^T$.

	\begin{figure}[!t]
		\begin{center}
			\includegraphics[height=0.29\textwidth]{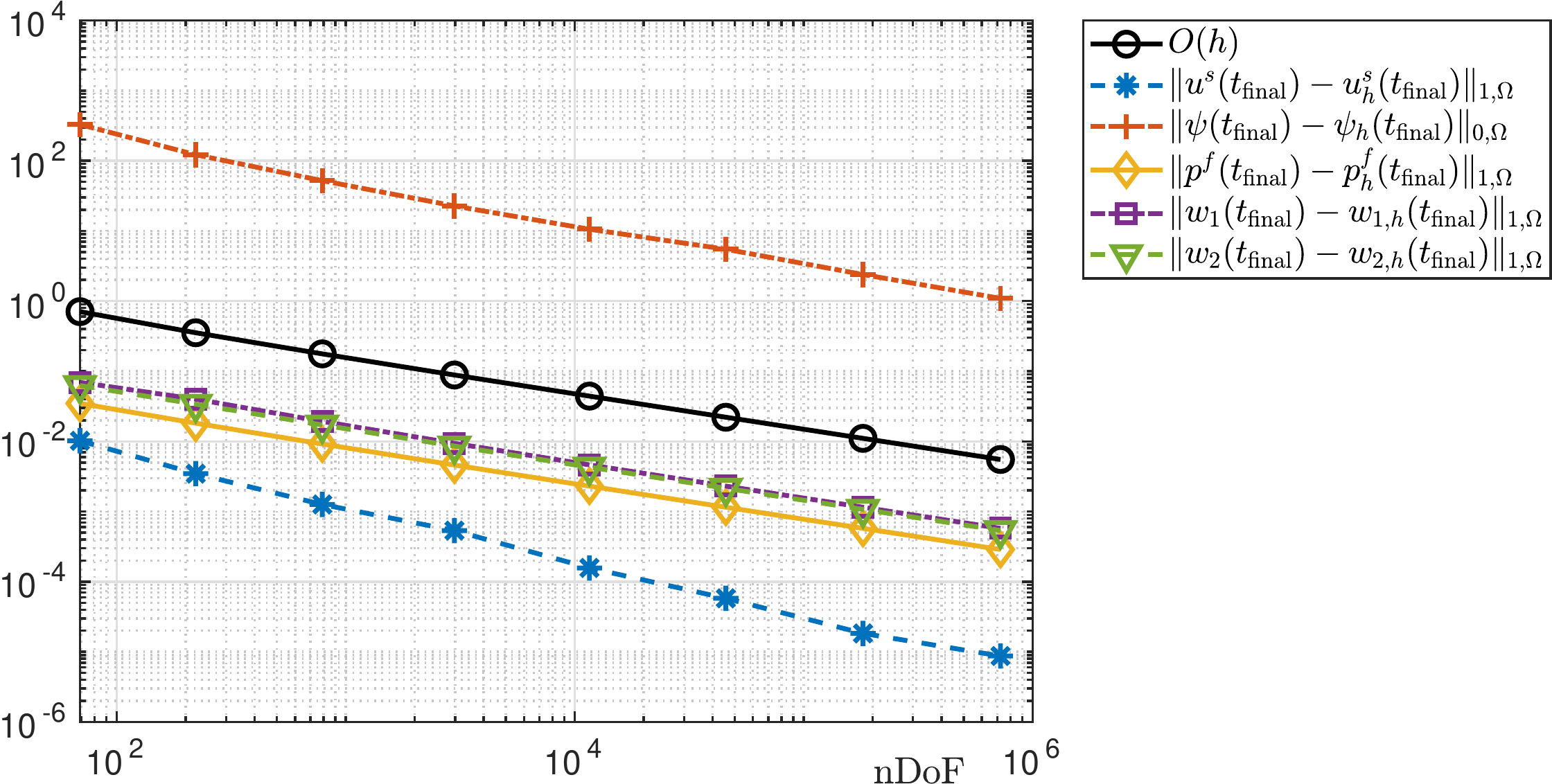}
			\includegraphics[height=0.29\textwidth]{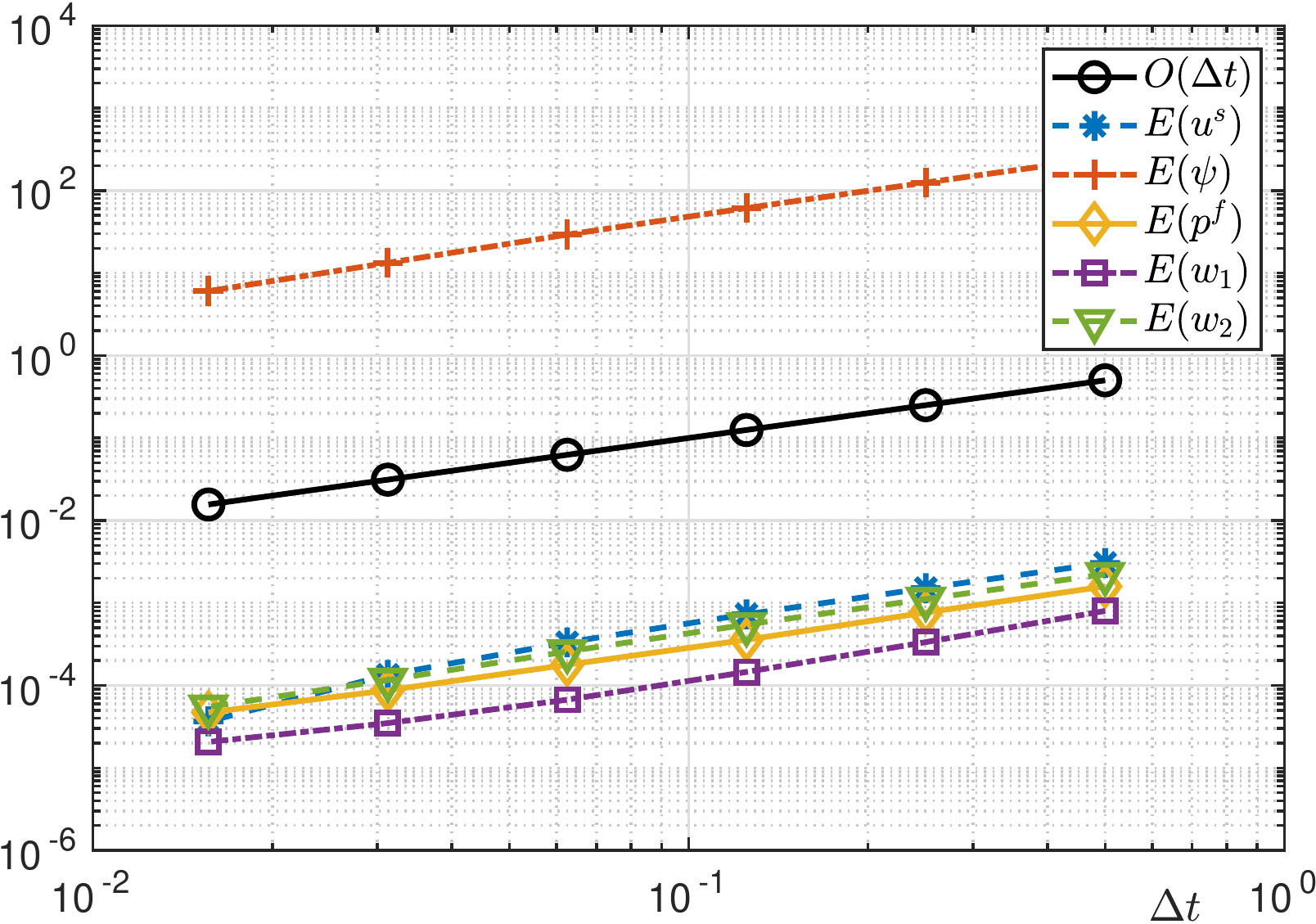}
		\end{center}
		\vspace{-3mm}
		\caption{Test 1. Convergence of the discretisation for the coupled poro-mechano-chemical
			problem. Error decay in space (left) and error history in time (right, where errors are computed from \eqref{eq:errors-dt}).}\label{fig:ex01}
	\end{figure}

	To \cblue{confirm numerically} the spatial accuracy of the discretisation defined by the finite element spaces
	specified in \eqref{eq:FEspaces}, we construct a sequence
	of seven uniformly refined meshes and compute individual
	approximate errors $e(\cdot)$ for each field in their natural spatial norm
	at the final time $t_{\text{final}}=0.04$, and
	the time-stepping scheme (backward Euler and implicit centred differences for first and second order time derivatives, respectively)
	approximates the polynomial dependence on time in \eqref{eq:sol-ex} exactly.
	The system is solved by the GMRES Krylov solver with
	incomplete LU factorisation (ILUT) preconditioning. The stopping criterion on the nonlinear iterations is based on a weighted residual norm dropping below the fixed
	tolerance of $1\cdot 10^{-5}$.
	Moreover, a small fixed time step $\Delta t =0.01$ is used for all mesh refinements. An average number of three Newton iterations are needed in all levels to reach convergence.
	The results are laid out in Figure \ref{fig:ex01} (left)
	where we observe an optimal error decay of $O(h)$ for all field variables.  We also see that the total error is dominated
	by the total pressure (which is large as these errors are not normalised and since the regime is nearly incompressible), but
	the convergence rates remain optimal with respect to the expected accuracy given by the interpolation properties
		of the finite element spaces \cblue{and stated in Theorem~\ref{theo:error}.}

	The convergence associated with the time discretisation can be more conveniently
	assessed considering a different set of closed-form solutions
	defined on a fixed mesh with 4000 elements
	$$ \bu^s =u_{\infty} \sin(t)\begin{pmatrix}\frac{x_1^2}{2\lambda}+{x_2^2}\\
	x_1^2+\frac{x_2^2}{2\lambda}\end{pmatrix}, \quad p^f = \sin(t)(x_1^2+x_1x_2),
	\quad w_1= \sin(t)(x_1^2-x_2^2), \quad w_2 =\sin(t)(x_1^2+x_2^2).$$
	With the given spatial discretisation, the errors will contain only contributions
	from the time approximation.
	We consider now the time interval $(0,1]$  and choose six time-step uniform
	refinements $\Delta t \in \{0.5,0.25,\ldots\}$ that we use to compute
	numerical solutions and cumulative errors up to $t_{\text{final}}$, of a generic individual field $s$ defined as
	\begin{equation}\label{eq:errors-dt}
	E(s) =  \biggl(\Delta t \sum_{n=1}^N \| s_h^n - s(t^n)\|_{0,\Omega}^2\biggr)^\frac{1}{2}.\end{equation}
	Figure \ref{fig:ex01} (right)
	indicates that the errors in time are also of first order, $O(\Delta t)$, \cblue{which also
	aligns with the convergence rates predicted by Theorem~\ref{theo:error}.}

	\begin{figure}[!t]
		\begin{center}
			\includegraphics[width=0.22\textwidth]{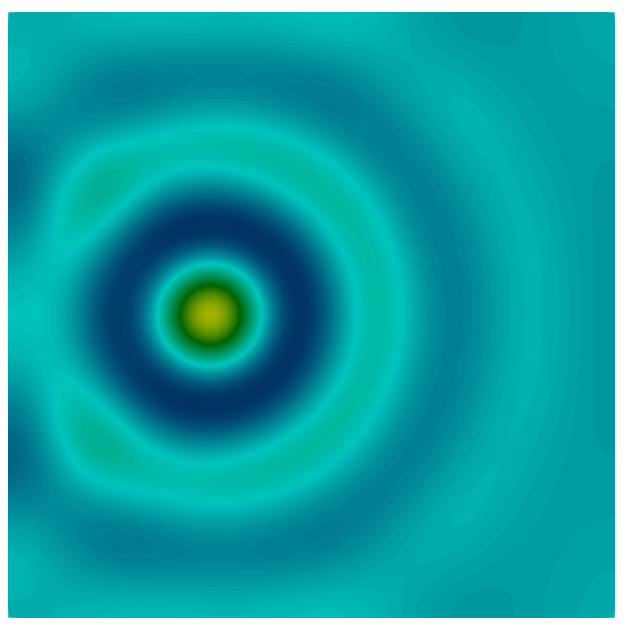}\
		\includegraphics[width=0.22\textwidth]{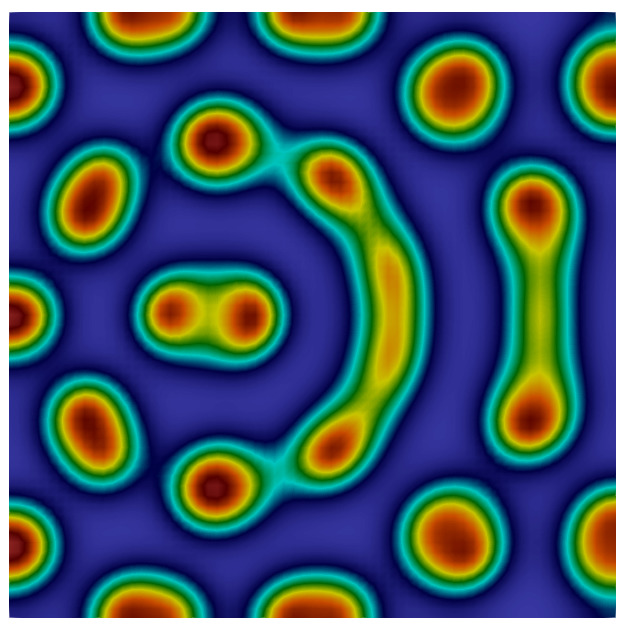}\
		\includegraphics[width=0.22\textwidth]{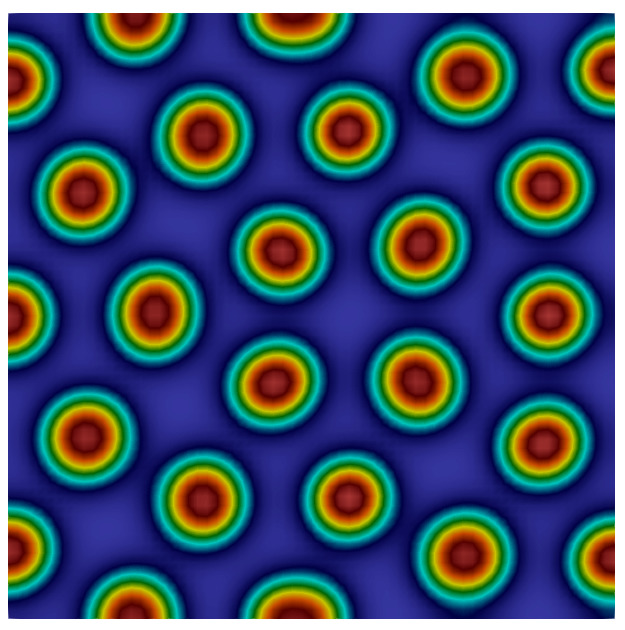}\\
\includegraphics[width=0.22\textwidth]{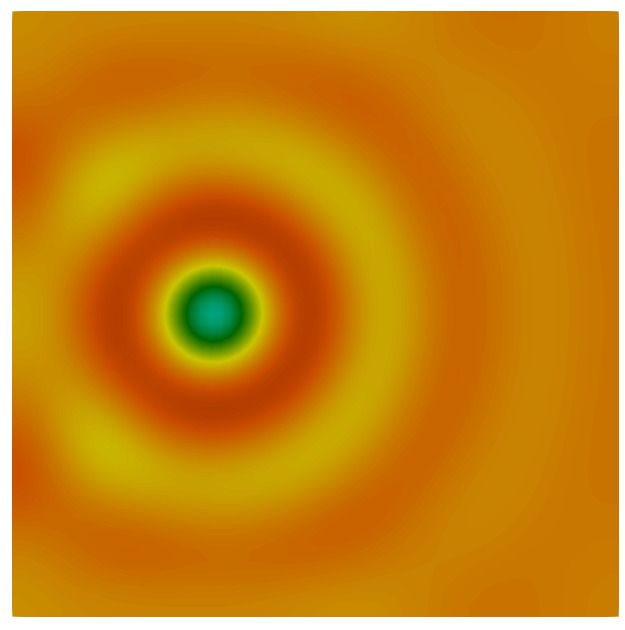}\
		\includegraphics[width=0.22\textwidth]{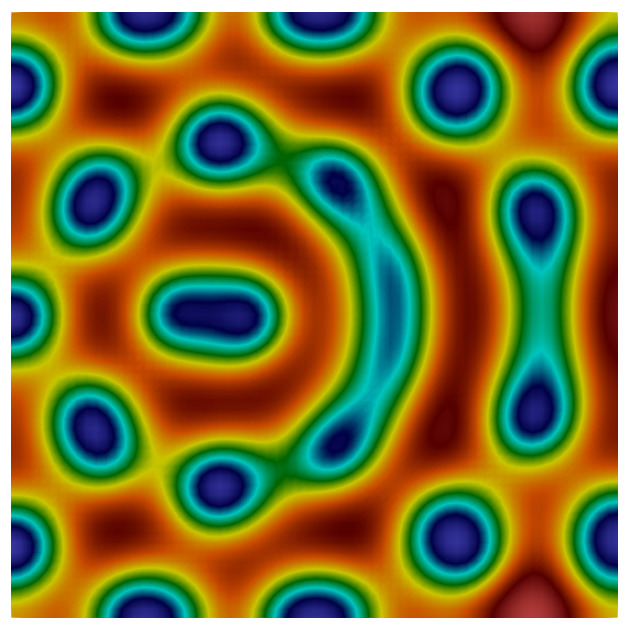}\
		\includegraphics[width=0.22\textwidth]{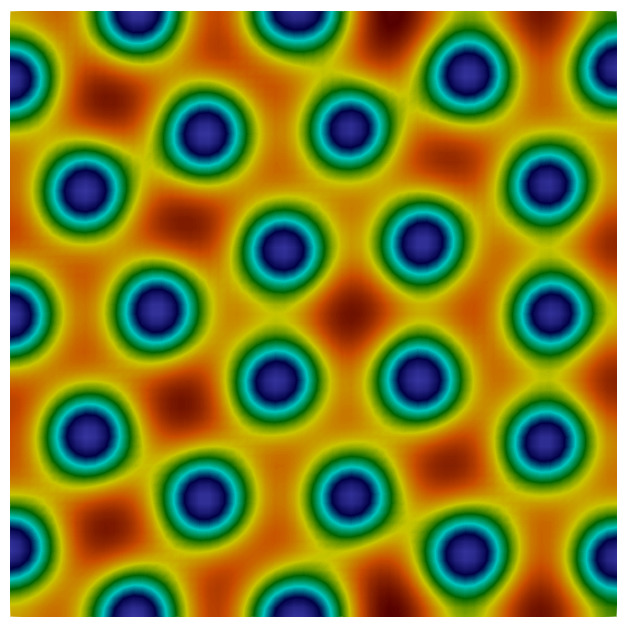}\\
		\includegraphics[width=0.245\textwidth]{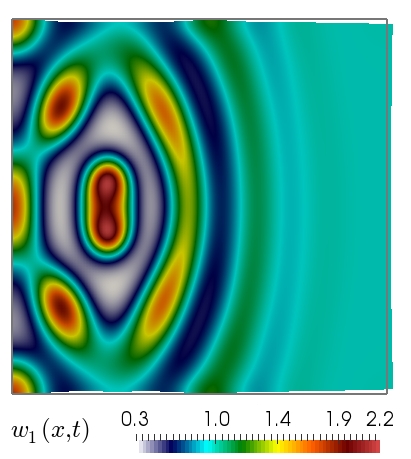}
		\includegraphics[width=0.245\textwidth]{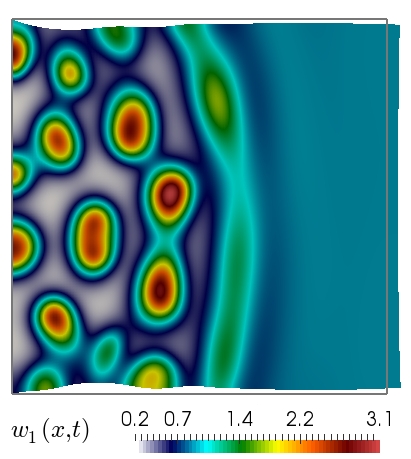}
		\includegraphics[width=0.245\textwidth]{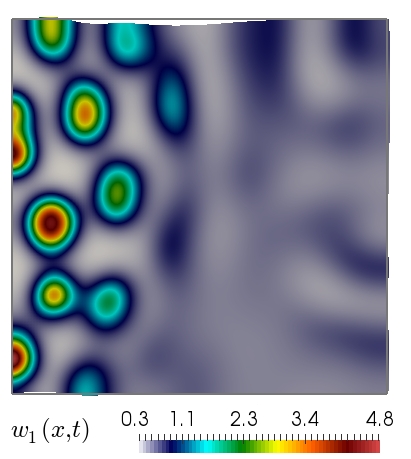}
		\includegraphics[width=0.245\textwidth]{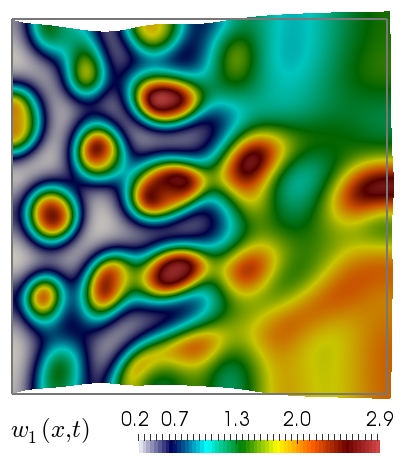}\\
		\includegraphics[width=0.245\textwidth]{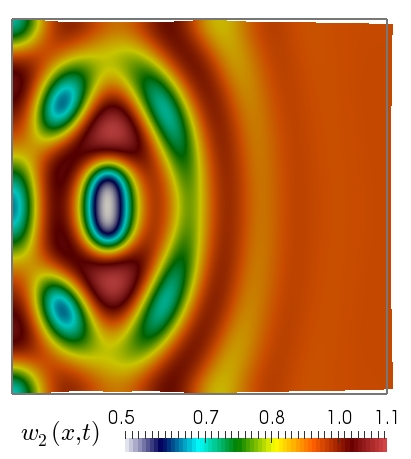}
		\includegraphics[width=0.245\textwidth]{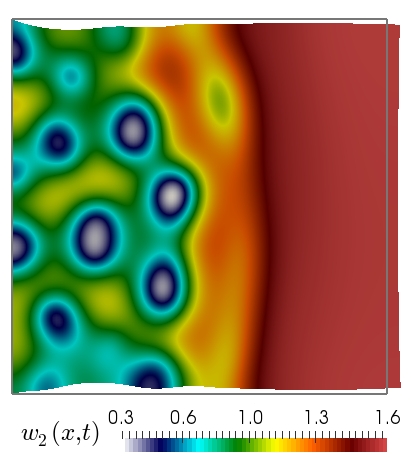}
		\includegraphics[width=0.245\textwidth]{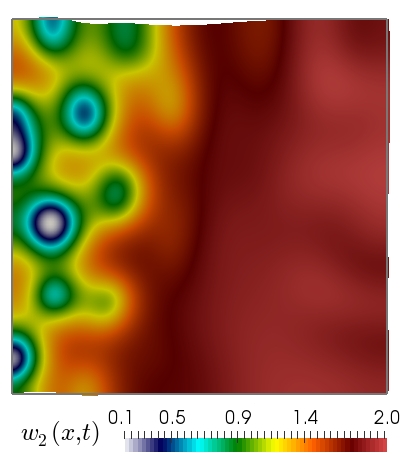}
		\includegraphics[width=0.245\textwidth]{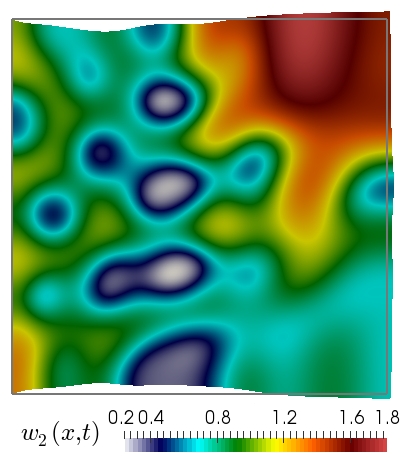}

		\end{center}
		\vspace{-3mm}
		\caption{\cblue{Test 2. Illustration of two-way coupling between poromechanical and
		chemical effects. Top rows: snapshots of concentrations of $w_1,w_2$ computed using $\gamma=0$, at three different times, reaching a stable state (right). Bottom rows: results obtained using $\gamma = 0.05$, and plotted on the deformed domain. These runs do not reach a stable spatial patterning, even after $t_{\text{final}}=10$.}}\label{fig:ex02}
	\end{figure}
	
\subsection{Example 2: Traction and active stress preventing stable patterning}
\cblue{Finally, we present a simple test to illustrate the application of the model and the
proposed
finite element method in the simulation of spatio-temporal chemical patterns. A
rectangular domain is considered $\Omega=(0,1)\times(0,0.6)$, where the right
segment constitutes the boundary $\Sigma$ on which a periodic-in-time traction is applied.
There we also impose zero fluid pressure. On the remainder of the boundary, $\Gamma = \partial\Omega \setminus\Sigma$ we prescribe zero
displacement and zero fluxes for the fluid. All parameters are taken as in Example 1, except for the active stress modulation $\tau=100$. The system is simulated until for the first round of computations $t_{\text{final}}=1$,
and we depict in Figure~\ref{fig:ex02} the approximate solutions. The top panels show
what the distribution of the chemical concentrations are when $\gamma = 0$ (that is, there is no two-way coupling as the chemicals are simply advected and diffused on the medium), and the distributions of the species are plotted on the reference, undeformed domain. Setting then $\gamma$ to a relatively small value $\gamma = 0.05$ modifies entirely the dynamics of the patterns. The periodic motion of the poroelastic slab and the chemically-induced active stress
imply that the stable state of the top panels is not reached (even if we continue towards
time horizons ten times longer than what we require in the first round of tests to achieve
a stable pattern). }

\cblue{We conclude this section summarising also our findings from \cite{verma02} dealing with the spectral linear stability analysis of the proposed model. We were able to demonstrate that the stability of the coupled system is influenced mainly by that of the special cases like homogeneous spatial distribution or uncoupled advection-diffusion-reaction sub-systems (\textit{i.e.}, $\tau=0$ and/or $\gamma=0$). We additionally observed that the strength of the coupling with poro-mechanical effects can bypass the conditions met by uncoupled sub-systems, and lead to linear instability and to the formation of complex spatio-temporal mechano-chemical patterns. For example, we have determined under which parameter regimes the system exhibits instability patterns. Also, a detailed derivation of the conditions leading to instabilities is outlined in the aforementioned reference.}	
	
	\section{Concluding remarks}\label{sec:concl}
	In this paper we have analysed a model of advection-reaction-diffusion in
	poroelastic materials. The set of equations assumes the regime of small strains and the coupling
	mechanisms are primarily dependent on source functions of change of volume, and active stresses. \cblue{All modelling aspects, implementation details for the mixed-primal scheme, application to biomedically-oriented problems, and a complete spectral stability analysis for the proposed system, can be found in our recent paper \cite{verma02}. In the present contribution we}
 have derived
	the well-posedness of the problem stated in mixed-primal form, and we have
	proposed a suitable mixed finite element scheme. Our work extends the similar-in-spirit contribution \cite{anaya18} in that
	we are able to derive stability bounds that are robust with respect to the Lam\'e constants of the solid. \cblue{Indeed, the main advantage of working with a mixed formulation for the
equations of poroelasticity is to have locking-free finite element schemes, which are of particular importance when the solids under consideration have large dilation modulus. These features are inherited from the method proposed in \cite{oyarzua16}, and a disadvantage with respect of adopting a formulation only in terms of displacement may be that we require more degrees of freedom.
	It is also noted that, since the proofs carried out here} do not rely entirely on the specific form of the reaction terms, the present formalism is quite general
	and could be applied to other systems with similar mathematical and physical structure, such as tumour development dynamics,
	long bones growth,
	or embryonic cell poromechanics.

As perspectives of this work, we aim at extending the analysis of Section~\ref{sec:solvability}
	to the case of finite-strain poroelasticity following the work in \cite{berger17}, \cblue{to
	cover also the effects of chemotaxis and general cross-diffusion, as well as interfacial conditions for two-layered materials \cite{anaya19b,deoliveira19,showalter02}}, and to incorporate viscoelasticity. Further directions include the design of
	mixed and double-mixed
	formulations that would improve the accuracy of the method in producing stresses or other variables of
	applicative interest and also contributing to achieve mass conservation \cite{kumar19,hong18}, as well as mesh adaptive methods guided by a posteriori error indicators \cite{ahmed19,ahmed20}.

\subsection*{Funding}
This work has been partially supported by CONICYT through the Becas-Chile Programme for foreign students, and by the London Mathematical Society through Scheme 5, Grant 51703.
	

\end{document}